\documentclass[11 pt]{amsart}

\usepackage[latin1]{inputenc}
\usepackage{amsmath,amsthm,amssymb,amscd,xcolor,hyperref}
\usepackage{bbm}
\usepackage{graphicx}
\usepackage{color}

\usepackage[notcite,notref]{showkeys}

\newtheorem{thm}{Theorem}[section]
 \newtheorem{cor}[thm]{Corollary}
 \newtheorem{lem}[thm]{Lemma}
 \newtheorem{prop}[thm]{Proposition}

 \theoremstyle{definition}
 
 \theoremstyle{remark}
 \newtheorem{rem}[thm]{Remark}
 
 \numberwithin{equation}{section}

\def\be#1 {\begin{equation} \label{#1}}
\newcommand{\ee}{\end{equation}}

\def\sqw{\hbox{\rlap{\leavevmode\raise.3ex\hbox{$\sqcap$}}$%
\sqcup$}}
\def\findem{\ifmmode\sqw\else{\ifhmode\unskip\fi\nobreak\hfil
\penalty50\hskip1em\null\nobreak\hfil\sqw
\parfillskip=0pt\finalhyphendemerits=0\endgraf}\fi}

\newcommand{\R}{\mathbb R}

\newcommand{\C}{\mathbb C}

\newcommand\<{\langle}
\renewcommand\>{\rangle}

\usepackage{cite}

\newcommand{\mc}{\mathcal}
\newcommand{\mbf}{\mathbf}

\DeclareMathOperator{\supp}{supp}

\DeclareMathOperator{\dist}{dist}

\newcommand{\X}{\mc H}

\newcommand{\bd}{\mbf d}

\usepackage{pgfplots,subcaption}

\title[Compactons and their variational properties]{Compactons and their variational properties for degenerate KdV and NLS in dimension 1}
\author{Pierre Germain}
\email{pgermain@cims.nyu.edu}
\address{Courant Institute of Mathematical Sciences, New York University\\
251 Mercer Street, New York, NY 10012, USA}
\thanks{P.G. was supported by the NSF grant DMS-15010.}
\author{Benjamin Harrop-Griffiths}
\email{benjamin.harrop-griffiths@cims.nyu.edu}
\address{Courant Institute of Mathematical Sciences, New York University\\
251 Mercer Street, New York, NY 10012, USA}
\thanks{B.H.-G. was supported by a Junior Fellow award from the Simons Foundation.}
\author{Jeremy L.~Marzuola}
\email{marzuola@math.unc.edu}
\address{Mathematics Department, University of North Carolina \\
Phillips Hall, Chapel Hill, NC 27599, USA}
\thanks{J.L.M. was supported in
  part by U.S. NSF Grants DMS--1312874 and DMS-1352353. }

\thanks{The authors thank John Hunter for very helpful conversations about degenerate NLS models at the beginning of the project.  Also, the authors are grateful to Rupert Frank for discussions related to energy minimizers, in particular bringing to our attention the reference~\cite{nagy1941integralungleichungen}. These discussions were initiated at the Mathematisches Forschungsinstitut Oberwolfach.}

\begin{document}

\maketitle

\begin{abstract}
 We analyze the stationary and traveling wave solutions to a family of degenerate dispersive equations of KdV and NLS-type.  In stark contrast to the standard soliton solutions for non-degenerate KdV and NLS equations, the degeneracy of the elliptic operators studied here allows for compactly supported steady or traveling states.  As we work in $1$ dimension, ODE methods apply, however the models considered have formally conserved Hamiltonian, Mass and Momentum functionals, which allow for variational analysis as well.
\end{abstract}

\section{Introduction}

\subsection{The classical theory} Before discussing the degenerate models that will be the focus of this article, it is good to have in mind the basic properties of the canonical one-dimensional models; classical references are~\cite{MR1696311, MR2233925}, see~\cite{MR695535} for ground states of more general semilinear one-dimensional models.

Consider the Hamiltonian (defined on functions on the real line)
$$
E(u) = \frac{1}{2} \int |\partial_x u|^2\,dx - \frac{1}{p} \int |u|^p\,dx.
$$
The associated Hamiltonian flows through the symplectic forms $(f,g) \mapsto \int \partial_x^{-1} f g \,dx$ and $(f,g) \mapsto \mathfrak{Im} \int \bar f g \,dx$ are, respectively, the generalized Korteweg-de Vries and nonlinear Schr\"odinger equations
\begin{align}
& \tag{KdV} \label{KdV} \partial_t u + \partial_x^3 u + \partial_x (u^{p-1}) = 0 \\
& \tag{NLS} \label{NLS} - i\partial_t u + \partial_x^2 u + |u|^{p-2} u = 0.
\end{align}
The profile $\phi$ of traveling waves of~\eqref{KdV} of the form $\phi(x-ct)$, or of stationary waves of~\eqref{NLS} of the form $\phi(x) e^{-ict}$   solves
$$
\partial_x^2 \phi - c \phi + \phi^{p-1} = 0.
$$
The only localized (decaying at infinity) solution of this equation is, up to translations,
$$
\phi(x) = c^{\frac{1}{2-p}} \psi \left( c^{\frac{p-3}{2-p}} x \right) \quad \mbox{with} \quad \psi(x) = \left( \frac{p}{2 \cosh\left( \frac{p-2}{2} x \right) ^2} \right)^{p-2}.
$$
Both ~\eqref{KdV} and~\eqref{NLS} conserve the $L^2$-mass $M(u) = \int |u|^2\,dx$ of the solution, and the above solutions $\phi$ can be viewed as critical points of the Hamiltonian $E$ under the constraint that the \(L^2\)-mass $M$ is fixed to some value. These critical points are actually global minimizers in the $L^2$-subcritical case $p < 5$, which leads to the orbital stability of these stationary waves.

Making use of the pseudo-Galilean invariance gives translating solutions of the nonlinear Schr\"odinger equation of the form \(\phi(x - vt)e^{-ict}e^{i(\frac12 xv + \frac14 tv^2)}\); they can also be characterized as minimizers of the Hamiltonian for fixed mass and momentum.

\subsection{A Hamiltonian leading to degenerate dispersion}

The aim of the present paper is to examine the situation if the Hamiltonian becomes 
$$
H(u) = \frac{1}{2} \int |u\partial_x u|^2\,dx - \frac{1}{p} \int |u|^p\,dx,
$$
making the equation quasilinear and degenerate close to $u=0$. We will assume throughout that
$$
p \geq 2.
$$
We will see that a number of interesting phenomena occur:
\begin{itemize}
\item Decaying stationary waves become compactly supported instead of exponentially decaying.
\item The family of stationary waves becomes two-dimensional (up to translations) instead of one-dimensional.
\item Out of these stationary waves, some, but not all, are energy minimizing for fixed mass.
\item The pseudo-Galilean symmetry has to be modified in a nonlinear way.
\end{itemize}

The degenerate KdV and NLS equations, obtained through the symplectic forms $(f,g) \mapsto \int \partial_x^{-1} f g \,dx$ and $(f,g) \mapsto \mathfrak{Im} \int \bar f g \,dx$ respectively, read
\begin{align}
\label{degKdV} \tag{dKdV}
& \partial_t u + \partial_x (u \partial_x (u \partial_x u) + u^{p-1}) = 0 \\
\label{degNLS} \tag{dNLS}
& - i \partial_t u + \bar u \partial_x ( u \partial_x u) + |u|^{p-2} u = 0. 
\end{align}
(where $u$ is real-valued for~\eqref{degKdV} and complex-valued for~\eqref{degNLS}). Both equations conserve the mass
$$
M(u) = \int |u|^2\,dx
$$
and additional conservation laws are given by
\begin{align*}
& \mbox{for~\eqref{degKdV},} \quad P(u) = \int u\,dx \\
& \mbox{for~\eqref{degNLS},} \quad  K(u) = \mathfrak{Im}\int \bar u \partial_x u \,dx.
\end{align*}

Degenerate KdV-type equations were first introduced and studied by Rosenau and Hyman~\cite{HR,MR1294558}; their primary interest was the existence of compactons. A Hamiltonian version of the Rosenau-Hyman equations was then proposed by Cooper, Shepard and Sodano~\cite{MR1376975, MR2787498}; a particular case of this family of equations is given by~\eqref{degKdV}. A Schr\"odinger version of the Rosenau-Hyman equations was proposed in~\cite{MR2468255}, of which~\eqref{degNLS} is a particular case.

The equations~\eqref{degKdV} and~\eqref{degNLS} are perhaps the simplest instances of degenerate dispersion; more elaborate models involving degenerate dispersion occur in the description of a variety of physical phenomena: to cite a few~\cite{MR2772627, MR2285103, Nesterenko, MR1234453, MR1135995, MR1811323, MR2599457}. It is our hope that the analysis of the model equations~\eqref{degKdV} and~\eqref{degNLS} will be an interesting step in the development of the mathematical theory of degenerate nonlinear dispersive equations, which remains very primitive.

\subsection{Obtained results}

\subsubsection{Compactons for~\eqref{degKdV}} Traveling waves of~\eqref{degKdV} are given by the ansatz $u(t,x) = \phi(x-ct)$ and satisfy the ODE
$$
- c \phi' + (\phi(\phi \phi')' + \phi^{p-1})' = 0.
$$
The analysis of this ODE leads in particular to the following theorem:

\begin{thm}
For $B>0$ and \(c\in \R\), or $B=0$ and $c>0$, there exist compactons $\Phi_{B,c}$ solving the above ODE (in the sense of distributions) which are even, compactly supported on $(-x_{B,c},x_{B,c})$,  decreasing on $(0,x_{B,c})$, and satisfy
$$
-\frac{c}{2} \phi^2 + \frac{1}{2} (\phi \phi')^2 + \frac{1}{p} \phi^p = B. 
$$
These compactons can be combined to yield multi-compacton solutions
$$
\sum \epsilon_i \Phi_{B_i,c}(x-x_i),
$$
where $\epsilon_i = \pm 1$, and the $\Phi_{B_i,c}(x-x_i)$ have non overlapping supports.
\end{thm}

This theorem is proved in Section~\ref{sect:TravelingWaves}.

\subsubsection{Variational properties} A classical idea to generate traveling waves is to consider the minimizing problem
$$
\min H(u) \quad \mbox{subject to} \quad M(u) = M_0,
$$
where $M_0$ is a fixed, positive constant. Traveling waves obtained through this minimization procedure are orbitally stable (as long as the flow around them can be defined).

\begin{thm}
For $2<p<8$, the above minimization problem admits a minimizer, which is (up to translation) one of the $\Phi_{B,c}$. For $p=4$, the minimizer is $\Phi_{0,c}$ with $c = \frac{M_0}{\sqrt{2} \pi}$.
\end{thm}

\begin{rem} While the proof of the above theorem relies on a concentration compactness result, we learned from Rupert Frank of the reference~\cite{nagy1941integralungleichungen}, where Sz. Nagy derives the same result from a very clever use of elementary inequalities. Sz. Nagy is also able to identify the minimizers as $\Phi_{B,c}$ for any $p \in (2,8)$.
\end{rem}

\subsubsection{Compactons for~\eqref{degNLS}} Traveling waves of~\eqref{degNLS} are given by the ansatz $u(t,x) = Q(x-vt) e^{-ict}$; they satisfy the equation
$$
-c Q + iv Q' + \bar Q(QQ')' + |Q|^{p-2} Q = 0.
$$
\begin{thm}
For $B>0$ and \(c\in \R\), or $B=0$ and $c>0$, there exist compactons $Q_{B,c}^v$ satisfying the above ODE. They are given by
$$
Q_{B,c}^v (x) = 
\left\{
\begin{array}{ll}
\displaystyle \Phi_{B,c} (x) e^{i v \theta_{B,c}(x)} & \mbox{if $x \in (-x_{B,c},x_{B,c})$} \\
0 & \mbox{if $x \notin (-x_{B,c},x_{B,c})$}
\end{array}
\right.
$$
where
$$
\left\{
\begin{array}{l}
\theta_{B,c}(0) = 0 \\
\displaystyle \theta_{B,c}' = - \frac{1}{2 \Phi_{B,c}^2} \quad \mbox{if $x \in (-x_{B,c},x_{B,c})$}
\end{array}
\right.
$$
\end{thm}

This theorem is proved in Section~\ref{sectiondegenerateNLS}.

\subsubsection{The linearized problem for~\eqref{degKdV}} Linearizing~\eqref{degKdV} around one of the compacton traveling waves $\phi$ results in the equation 
$$
\partial_t v = \partial_x \mathcal{L}_\phi v \qquad \mbox{with} \qquad \mathcal{L}_\phi = - \phi (\partial_x^2 +2) \phi
$$
(in the moving frame). This equation has to be supplemented with appropriate boundary conditions.

We discuss in Section~\ref{sectionlinearstability}, in the case $p=4$,
\begin{itemize}
\item The spectrum of the operators $\mathcal{L}_\phi$, which depends on the compacton considered
\item A local well-posedness theory for the linearized evolution problem above.
\end{itemize}

\subsubsection{Numerical results} Numerical results on the traveling waves and their stability are given in Section~\ref{sectionnumericalanalysis}.

\section{Solitary waves for~\eqref{degKdV}}\label{sect:TravelingWaves}

\subsection{The ODE} By definition, traveling waves at velocity $c \in \mathbb{R}$ are solutions of the form
$$
u = \phi(x-ct).
$$
Inserting this ansatz in the equation leads to
\begin{equation}
\label{theODE}
-c \phi' + (\phi(\phi \phi')' + \phi^{p-1})' = 0.
\end{equation}
This is equivalent to
\begin{equation} \label{theODEA}
-c \phi + \phi(\phi \phi')' + \phi^{p-1} = A
\end{equation}
for an integration constant $A \in \mathbb{R}$. 
Multiplying by $\phi'$, we get for a new integration constant $B$
$$
- \frac{c}{2}\phi^2 + \frac{1}{2}(\phi \phi')^2 + \frac{1}{p} \phi^p = A\phi + B,
$$
or equivalently
\begin{equation}
\label{ODEAB}
(\phi')^2 = \frac{2B}{\phi^2} + \frac{2A}{\phi} + c - \frac{2}{p} \phi^{p-2} = F_{A,B,c} (\phi)
\end{equation}
The following proposition classifies solutions of this ODE:

\medskip
\begin{prop}
\label{propODE}
Assume $p > 2$. We will denote $z_{A,c}$ be the unique positive solution of $A + cz = z^{p-2}$, when it exists.

Consider a solution $\phi$ of~\eqref{theODE}, with speed $c$ and integration constants $A$ and $B$, which is positive over the (maximal) interval $I$. Up to translation, there are three possibilities
\begin{itemize} 
\item[(i)] $I = (-\infty,\infty)$, and the solution $\phi$ is periodic. Such a solution exists in particular if $B<0$, or $B=0$ and $A<0$, and, if $z_{A,c}$ exists, $F_{A,B,c}(z_{A,c}) \neq 0$. 
\item[(ii)] $I = (0,\infty)$. This is the case if and only $z_{A,c}$ exists, 
satisfies $F_{A,B,c}(z_{A,c}) =0$; if furthermore $F_{A,B,c}$ is positive on $[0,z]$; and if finally $B>0$, or $B=0$ and $A>0$, or $A=B=0$ and $c>0$. Then
$$
\phi(x) \to z \quad \mbox{as $x \to \infty$},
$$
and the behavior of $\phi$ for $x \to 0$ is as in the following point.
\item[(iii)] $I = (-X,X)$ for some $X>0$, and the solution $\phi$ is even. This is the case if and only if one of the two following conditions is satisfied:
\begin{itemize}
\item[(1)] $B>0$ and $A \neq 0$, or $B=0$ and $A>0$, and furthermore $F_{A,B,c}(z_{A,c}) \neq 0$ if $z_{A,c}$ exists. Then, as $x \to 0$,
\begin{align*}
\phi(-X+x) \sim 
\left\{
\begin{array}{ll}
\displaystyle \sqrt{2\sqrt{2B}} \sqrt x + \frac{2A}{3\sqrt{2B}} x + O(x^{3/2}) & \mbox{if $B>0$, $A \neq 0$} \\
\displaystyle \frac{3^{2/3}A^{1/3}}{2^{1/3}} x^{2/3} - \frac{3^{4/3}c}{10 \sqrt{2}} x^{4/3} + O(x^2) & \mbox{if $B=0$, $A>0$}.
\end{array}
\right.
\end{align*}
\item[(2)] $A=0$ and $B>0$, or $A=B=0$, $c>0$. Then, as $x \to 0$,
$$
\phi(-X+x) = 
\left\{ 
\begin{array}{ll}
\displaystyle \sqrt{2 \sqrt{2B}} \sqrt{x} + \frac{c}{2^{3/2}(2B)^{1/4}}  x^{3/2} + O(x^{2}) & \mbox{if $A=0$ and $B>0$} \vspace{0.1cm}\\
\displaystyle \sqrt{c} x + O(x^3) & \mbox{if $A=B=0$ and $c>0$}.
\end{array}
\right.
$$
\end{itemize}
\end{itemize}
\end{prop}

\begin{rem} The case $p=2$ behaves differently from $p>2$, and has therefore not been included in the above. We refer to Section~\ref{subsectionexplicit} for a description of compactons in this case.
\end{rem}

\begin{proof} Recall that, once the integration constants are fixed, $\phi$ satisfies $(\phi')^2 = F_{A,B,c}(\phi)$. Thus, for $A$, $B$ and $c$ fixed, we think of the phase space as foliated by sets of the type $\{ (\phi')^2 = F_{A,B,c}(\phi) \}$.
\begin{figure}[h]
\centering
\includegraphics[width=0.35\textwidth]{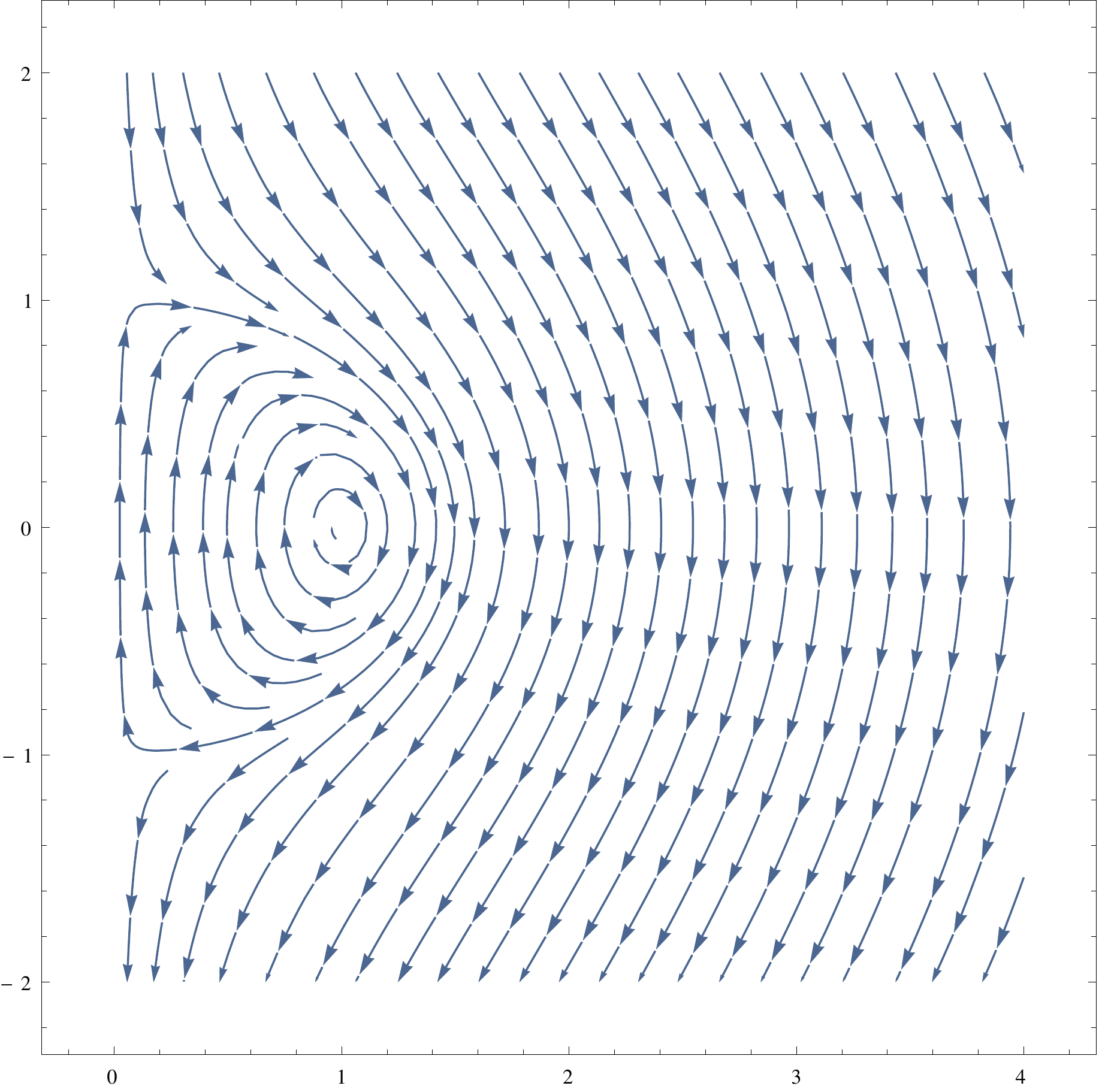}
\caption{The phase portrait for $c=1$, $p=4$, and $A=0$.}
\end{figure}
If the set $\{ (\phi')^2 = F_{A,B,c}(\phi) \}$ is bounded, we are in the situation of $(i)$. 

If it is unbounded, we are either in the situation $(ii)$, or $(iii)$. Situation $(ii)$ corresponds to the case where the stable point $z_{A,c}$ satisfies $F_{A,B,c}(z_{A,c}) = 0$. Notice that for $A=0$, $z_{0,c}$ exists for $c>0$, and satisfies then $(z_{0,c})^{p-1} = c$. Then 
$$
F_{A,B,c}(z_{0,c}) = 0 \quad \Leftrightarrow \quad \frac{2B}{(z_{0,c})^2} = c \left( \frac{2}{p} -1 \right).
$$
Simply comparing signs, we see that the above inequality cannot hold for $B \geq 0$, so that $F_{0,B,c}(z_{0,c}) \neq 0$ for $A=0$, $B \geq 0$.

In order to establish the behavior close to $\phi=0$ in the case $A=0$, $B>0$ (the other cases being similar), we need to integrate locally the ODE
$$
\phi' = \sqrt{\frac{2B}{\phi^2} + c - \frac{2}{p}\phi^{p-2}}=g(\phi).
$$
Expanding $g$, we obtain that 
\begin{align*}
x & = \int_0^\phi \frac{ds}{g(s)} = \int_0^\phi \left[\frac{s}{\sqrt{2B}} - \frac{cs^3}{4B\sqrt{2B}} + O(s^4) \right]\,ds = \frac{\phi^2}{2\sqrt{2B}} - \frac{c \phi^4}{16 B \sqrt{2B}} + O(\phi^5),
\end{align*}
from which we deduce the desired expansion
$$
\phi(x) = \displaystyle \sqrt{2 \sqrt{2B}} \sqrt{x} + \frac{c}{2^{3/2}(2B)^{1/4}} x^{3/2} + O(x^{2}).
$$

We next consider the point $y$ at which $F_{A,B,c}$ vanishes first. If this point is also a stationary point of the phase portrait, i.e. $A + cy = y^{p-2}$, then we are in case $(ii)$. Otherwise, we can prolong the solution by symmetry to obtain case $(iii)$.
\end{proof}

\subsection{The compactons}

We will mostly focus on finite mass (or, equivalently, finite energy) traveling waves. We will denote $\Phi_{B,c}$ for the maximal positive solution corresponding to~$(iii)$ in Proposition~\ref{propODE} with $A=0$, $B>0$, and $c\in \mathbb{R}$; or $A=B=0$, and $c>0$, prolonged by $0$ outside of its domain of positivity, which we denote $(-x_{B,c},x_{B,c})$. Recapitulating,
\begin{itemize}
\item $\Phi_{B,c}$ even.
\item $\Phi_{B,c} >0$ and smooth on $(-x_{B,c},x_{B,c})$.
\item $\Phi_{B,c} = 0$ on $(-\infty,-x_{B,c}) \cup (x_{B,c},\infty)$.
\item $\Phi_{B,c}$ decreasing on $(0,x_{B,c})$. 
\end{itemize}

Let us check that $\Phi_{B,c}$ satisfies~\eqref{theODE} in the sense of distributions on $\mathbb{R}$. This is implied by the equality 
$$
\mbox{for $\phi = \Phi_{B,c}$}, \quad -c \phi + \phi( \phi \phi')' + \phi^{p-1} = 0
$$
in the sense of distributions. The only delicate points are of course $\pm x_{B,c}$; more precisely, it is easy to see that $(\phi \phi')'$ is the sum of a bounded function and of $\delta_{\pm x_{B,c}}$. But since $\phi$ vanishes at $\pm x_{B,c}$, the product $\phi( \phi \phi')'$ is a bounded function; from there it is easy to check that the above holds.

Notice that, defining $\Phi_{A,B,c}$ in a similar way to $\Phi_{B,c}$, it does not satisfy~\eqref{theODE} in the sense of distributions on $\mathbb{R}$ but rather
$$
\mbox{for $\phi = \Phi_{A,B,c}$}, \quad -c \phi + \phi( \phi \phi')' + \phi^{p-1} = A \delta_{-x_{A,B,c}} - A \delta_{x_{A,B,c}}.
$$

This leads to the following proposition:

\begin{cor}
For a velocity $c \in \mathbb{R}$, general solutions of finite mass of~\eqref{theODE} (in the sense of distributions) are of the form
$$
\sum \epsilon_i \Phi_{B_i,c}(x-a_i),
$$
where $a_i \in \mathbb{R}$, either $B_i>0$ or $c>0$, $\epsilon_i = \pm 1$, and the supports of the $\Phi_{B_i,c}(x-a_i)$ are disjoint.
\end{cor}

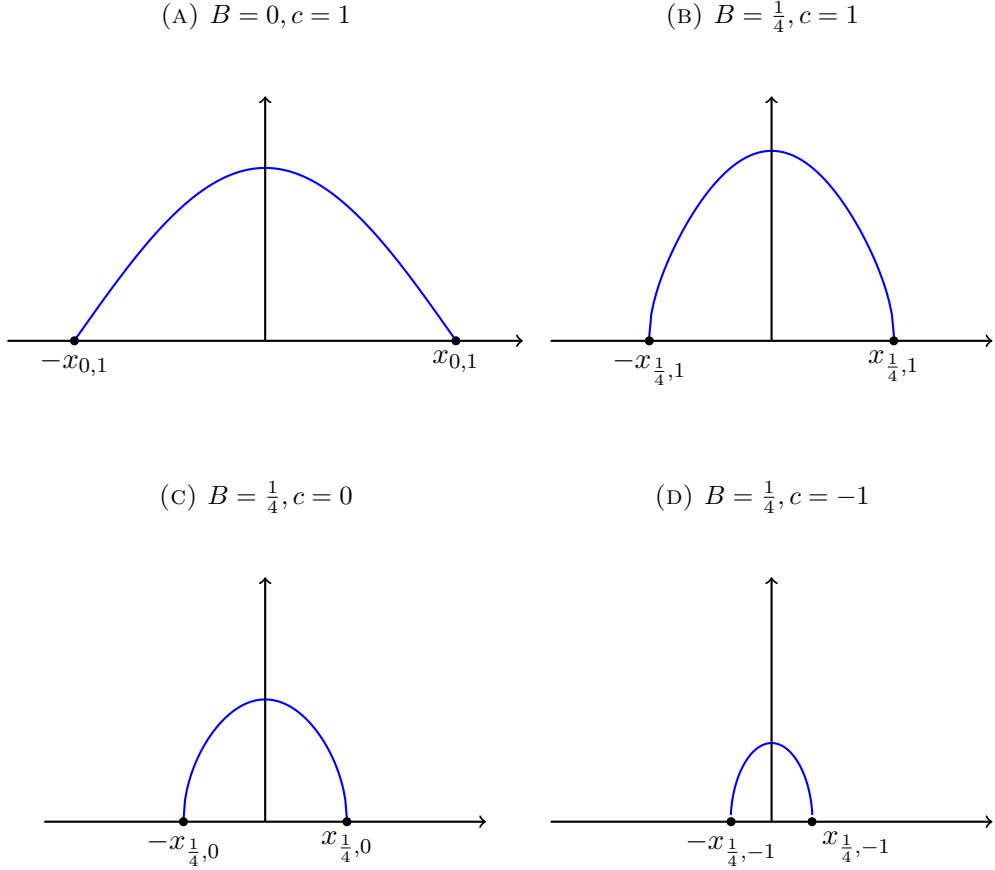
\begin{figure}[h]
\begin{subfigure}[b]{0.4\textwidth}
\caption{\(B = 0,c = 1\)}
\centering

  \begin{tikzpicture}
    \begin{axis}[ xmin=-3, xmax=3, ymin=-1, ymax=2.5, axis x
      line = none, axis y line = none, samples=100 ]

      \addplot[mark=none,domain=-pi/sqrt(2):pi/sqrt(2),thick,blue] {sqrt(1 + cos(deg(sqrt(2)*x)))};

      \draw[black,->,thick] (axis cs:0,0) -- (axis cs:0,2);
      \draw[black,->,thick] (axis cs:-3,0) -- (axis cs:3,0);
      
      \filldraw (axis cs:-2.22,0) circle (1.5pt) node[below] {\(-x_{0,1}\)};
      \filldraw (axis cs:2.22,0) circle (1.5pt) node[below] {\(x_{0,1}\)};

    \end{axis}
  \end{tikzpicture}
\end{subfigure}
\begin{subfigure}[b]{0.4\textwidth}
\caption{\(B = \frac14,c = 1\)}
\centering

  \begin{tikzpicture}
    \begin{axis}[ xmin=-3.5, xmax=3.5, ymin=-1, ymax=2.5, axis x
      line = none, axis y line = none, samples=100 ]

      \addplot[mark=none,domain=-3*pi/(4*sqrt(2)):3*pi/(4*sqrt(2)),thick,blue] {sqrt(1 + sqrt(2)*cos(deg(sqrt(2)*x)))};

      \draw[black,->,thick] (axis cs:0,0) -- (axis cs:0,2);
      \draw[black,->,thick] (axis cs:-3,0) -- (axis cs:3,0);
      
      \filldraw (axis cs:-1.66,0) circle (1.5pt) node[below] {\(-x_{\frac14,1}\)};
      \filldraw (axis cs:1.66,0) circle (1.5pt) node[below] {\(x_{\frac14,1}\)};

    \end{axis}
  \end{tikzpicture}
\end{subfigure}

\begin{subfigure}[b]{0.4\textwidth}
\caption{\(B = \frac14,c = 0\)}
\centering

  \begin{tikzpicture}
    \begin{axis}[ xmin=-3.5, xmax=3.5, ymin=-1, ymax=2.5, axis x
      line = none, axis y line = none, samples=100 ]

      \addplot[mark=none,domain=-pi/(2*sqrt(2)):pi/(2*sqrt(2)),thick,blue] {sqrt( cos(deg(sqrt(2)*x)))};

      \draw[black,->,thick] (axis cs:0,0) -- (axis cs:0,2);
      \draw[black,->,thick] (axis cs:-3,0) -- (axis cs:3,0);
      
      \filldraw (axis cs:-1.11,0) circle (1.5pt) node[below] {\(-x_{\frac14,0}\)};
      \filldraw (axis cs:1.11,0) circle (1.5pt) node[below] {\(x_{\frac14,0}\)};

    \end{axis}
  \end{tikzpicture}
\end{subfigure}
\begin{subfigure}[b]{0.4\textwidth}
\caption{\(B = \frac14,c = -1\)}
\centering

  \begin{tikzpicture}
    \begin{axis}[ xmin=-3.5, xmax=3.5, ymin=-1, ymax=2.5, axis x
      line = none, axis y line = none, samples=500 ]

      \addplot[mark=none,domain=-pi/(4*sqrt(2)):pi/(4*sqrt(2)),thick,blue] {sqrt( -1 + sqrt(2)*cos(deg(sqrt(2)*x)))};

      \draw[black,->,thick] (axis cs:0,0) -- (axis cs:0,2);
      \draw[black,->,thick] (axis cs:-3,0) -- (axis cs:3,0);
      
      \filldraw (axis cs:-0.55,0) circle (1.5pt) node[below] {\(-x_{\frac14,-1}\)};
      \filldraw (axis cs:0.55,0) circle (1.5pt) node[below right] {\(x_{\frac14,-1}\)};

    \end{axis}
  \end{tikzpicture}
\end{subfigure}
\caption{Compactons for \(A = 0, p = 4\)}
\end{figure}

\subsection{Explicit formulas}

\label{subsectionexplicit}

If $p=2$, $A=0$, the equation~\eqref{theODEA} becomes
\begin{equation}
\phi (1-c) + \phi ( \phi \phi')' = 0.
\end{equation}
Setting $\rho = \phi^2$, this becomes
\begin{equation}
\rho'' = 2 (c-1).
\end{equation}
Assuming that $\rho$ is even, this can be integrated to give that $\rho =  (c-1) x^2 + Y$, where $Y = \frac{2B}{1-c}$. This gives the compacton
$$
\mbox{if $c<1$, $B \geq 0$}, \quad \Phi_{B,c} =
\left\{
\begin{array}{cl}
\sqrt{ (c-1) x^2 + \frac{2B}{1-c} } & \mbox{if $|x| < \frac{\sqrt{2B}}{1-c}$} \\
0 & \mbox{otherwise}
\end{array}
\right.
$$

\bigskip

If $p=4$, $A=0$, the equation for $\rho = \phi^2$ becomes linear:
$$
\frac{1}{2}\rho'' + \rho = c.
$$
It can easily be integrated to yield that $\rho = c + Z \cos(\sqrt 2 x)$ with $4B = Z^2 - c^2$, which leads to three cases:
\begin{itemize}
\item If $c>0$ and $0 > 4B>-c^2$, then periodic solutions are given by
$$
\phi(x) = \sqrt{c + \sqrt{4B+c^2} \cos(\sqrt{2} x)} .
$$
\item If $c \in \mathbb{R}$ and $B>0$, then compactons are given by
\begin{equation}
\label{formulacompacton}
\Phi_{B,c}(x) =
\left\{
\begin{array}{ll}
\sqrt{c + \sqrt{4B+c^2} \cos(\sqrt{2} x)} & \mbox{if $x \in [-x_{B,c},x_{B,c}]$} \\
0  & \mbox{if $x \notin [-x_{B,c},x_{B,c}]$},
\end{array}
\right.
\end{equation}
where $x_{B,c}$ is the least positive solution of $\cos(\sqrt 2 x) = - \frac{c}{\sqrt{4B+c^2}}$.
\item If $c>0$ and $B=0$, a compacton is given by
$$
\Phi_{0,c}(x) =
\left\{
\begin{array}{ll}
\sqrt{2c}\cos \left(\frac{x}{\sqrt 2}\right) & \mbox{if $x \in [-\frac{\pi}{\sqrt 2},\frac{\pi}{\sqrt 2}]$} \\
0  & \mbox{if $x \notin [-\frac{\pi}{\sqrt 2},\frac{\pi}{\sqrt 2}]$},
\end{array}
\right.
$$
\end{itemize}

\subsection{The energy and mass of compactons}
Notice first the scaling property: for any $\lambda>0$,
$$
\lambda \Phi_{B,c} (\lambda^{\frac{p}{2}-2} \cdot) = \Phi_{B\lambda^p,c \lambda^{p-2}}.
$$
Next, we record how the Hamiltonian, Mass, and support of $\Phi_{B,c}$ are related:

\medskip
\begin{prop} 
\label{propenergypohozaev}
The compacton $\Phi = \Phi_{B,c}$ satisfies
$$
H(\Phi) = c \left( \frac{p-8}{2p+8} \right) M(\Phi) + \left( \frac{p-4}{p+4} \right) 2B x_{B,c}.
$$
\end{prop}

\begin{proof} The desired formula follows by combining the energy identity
$$
c \int \Phi^2 +2 \int (\Phi \Phi')^2 - \int \Phi^p = 0,
$$
which can be obtained by multiplying~\eqref{theODEA} with $A=0$ by $\Phi$, and integrating by parts, making sure that the boundary terms vanish, with the Pohozaev identity
$$
- c \int \Phi^2 + \int (\Phi \Phi')^2 + \frac{2}{p} \int \Phi^p = 4B x_{B,c}
$$
(which can be obtained by multiplying~\eqref{theODEA} with $A=0$ by $x\Phi'$, and integrating; or equivalently, by integrating~$- c \Phi^2 + \Phi^2 (\Phi')^2 + \frac{2}{p} \Phi^p = 2B$ over $[-x_{B,c},x_{B,c}]$).

\end{proof}
\medskip

\begin{prop} \label{propminH}
If $p=4$, the minimum of $H(\Phi_{B,c})$ given $M(\Phi_{B,c}) = m$ is reached for $B=0$ and $c = \frac{m}{\sqrt 2 \pi}$. The value of the minimum is $\displaystyle - \frac{m^2}{4\sqrt 2 \pi}$.
\end{prop}
\begin{proof} From Proposition~\eqref{propenergypohozaev}, we learn if $p=4$ that
$$
H(\Phi_{B,c}) = - \frac{c}{4}  M(\Phi_{B,c}).
$$
This implies first that the minimum of $H$ is reached for $c>0$, which we assume from now on; equivalently, $x_{B,c} \in [\frac{\pi}{2\sqrt 2},\frac{\pi}{\sqrt 2}]$. Next, since $M(\Phi_{B,c})$ is fixed at $m$, the above becomes
$$
H(\Phi_{B,c}) = - \frac{c}{4} m.
$$
Thus our aim becomes: find the the compacton with mass $m$ and the largest speed $c$. An easy computation using the formula~\eqref{formulacompacton} gives
$$
m = M(\Phi_{B,c}) = c ( 2 x_{B,c} - \sqrt{2} \tan ( \sqrt{2} x_{B,c} )).
$$
Since the function $y \mapsto 2y - \sqrt{2} \tan(\sqrt 2 y)$ is decreasing on $[\frac{\pi}{2\sqrt 2},\frac{\pi}{\sqrt 2}]$, we find that the largest value of $c$ is reached when $x_{B,c} = \frac{\pi}{\sqrt 2}$, or equivalently $B=0$.
\end{proof}

\begin{prop} If $p=2$, there is no compacton $\Phi_{B,c}$ which achieves the minimum of $H(\Phi_{B,c})$ for $M(\Phi_{B,c})$ fixed.
\end{prop}
\begin{proof}
A simple computation gives that
$$
M(\Phi_{B,c}) = \frac{4}{3} \frac{(2B)^{3/2}}{(1-c)^2}\quad \mbox{while} \quad H(\Phi_{B,c}) = - \frac{(2B)^{3/2}}{3} \frac{1+c}{(1-c)^2} = -\frac{1+c}{4} M(\Phi_{B,c}),
$$
so that the infimum is achieved for $c \to 1$; but $c=1$ is not allowed for the compactons $\Phi_{B,c}$.
\end{proof}

\subsection{Variational analysis for $p<8$}

\label{subsectionvariationaldKdV}

\begin{thm}[Ground state] If $2<p<8$, then for any $m>0$, $\displaystyle \inf_{M(u) = m} H(u) > -\infty$, and this variational problem admits minimizers. Modulo translation, they are of the form $\pm \Phi_{B,c}$ with $M(\Phi_{B,c}) = m$.

If $p=4$, the minimizing compacton is
$$
\Phi_{0,c}(x), \quad \mbox{with} \; \;c = \frac{m}{\sqrt 2 \pi}.
$$
\end{thm}

\begin{rem} It is probably the case that the minimizer is given by $\Phi_{0,c}$ for any $p \in (2,8)$.
\end{rem}

\begin{proof} Switching to the unknown function $\rho = u^2$, the problem becomes that of minimizing
$$
\mathcal{H}(\rho) = \frac{1}{8} \int (\rho')^2 - \frac{1}{p} \int \rho^{\frac{p}{2}}, \quad \mbox{subject to} \quad \mathcal{M}(\rho) = \int \rho\,dx = m \quad \mbox{and} \quad \rho \geq 0.
$$

The Gagliardo-Nirenberg inequality
\begin{equation}
\label{L1L2GNineq}
\| \rho \|_{L^{\frac{p}{2}}} \lesssim \| \rho \|_{L^1}^{\frac{1}{3} + \frac{4}{3p}} \| \rho' \|_{L^2}^{\frac{2}{3} \left( 1 - \frac{2}{p} \right)}
\end{equation}
implies that
$$
I_m = \inf_{\mathcal{M}(\rho) = m} \mathcal{H} (\rho) \gtrsim \inf \| \rho' \|_{L^2}^2 - \| \rho' \|_{L^2}^{\frac{2}{3} \left( \frac{p}{2} -1 \right)}  > -\infty.
$$
Furthermore, $I_m <0$ since, for any $\rho_0 \in \mathcal{C}_0^\infty$, $\rho_0 \geq 0$, and any $\lambda>0$, we have $\mathcal{M}(\lambda \rho_0 (\lambda \cdot)) = \mathcal{M}(\rho_0)$ while
$$
\mathcal{H}(\lambda \rho_0 (\lambda \cdot)) = \frac{\lambda^3}{8} \int (\rho_0')^2 - \frac{\lambda^{\frac{p}{2}-1}}{p} \int \rho^{\frac{p}{2}} < 0 \quad \mbox{for $\lambda$ sufficiently small}.
$$
Finally, observe the scaling law:
$$
I_m = m^{\frac{p+4}{8-p}} I_1.
$$
It suffices to show the existence of a minimizer for $m=1$, so we consider a minimizing sequence $\rho_n$ for $\mathcal{H}$ such that $\rho_n \geq 0$, $M(\rho_n) =1$. We now apply the following concentration compactness result

\begin{prop}
Consider $(\rho_n)$ a bounded sequence of nonnegative functions in $H^1(\mathbb{R}) \cap L^1(\mathbb{R})$. Then there exists a subsequence (still denoted $(\rho_n)$), a family of sequences $(x_n^j)$, and functions $(V^j)$ such that, defining furthermore $\rho_n^J$ by
$$
\rho_n(x) = \sum_{j=1}^J V^j (x - x_n^j) + \rho_n^J,
$$
there holds
\begin{itemize}
\item $\forall k \neq j \in \mathbb{N}$, $|x_n^k - x_n^j| \to \infty$ as $n \to \infty$;
\item $\forall j \in \mathbb{N}$, $V^j \geq 0$ and $V^j \in H^1 \cap L^1$;
\item $\limsup_{n \to \infty} \| \rho_n^J \|_{L^q} \to 0$ as $J \to \infty$ for all $q \in (1,\infty)$;
\item $\| \partial_x \rho_n \|_{L^2}^2 - \sum_{j=1}^J \| \partial_x V^j \|_{L^2}^2 - \| \partial_x \rho_n^J \|_{L^2}^2 \to 0$ as $n \to \infty$;
\item $\| \rho_n \|_{L^q}^q - \sum_{j=1}^J \| V^j \|_{L^q}^q - \| \rho_n^J \|_{L^q}^q \to 0$ as $n \to \infty$ for all $q \in [1,\infty)$.
\end{itemize}
\end{prop}

The proof follows mutatis mutandis from the proof of Proposition 3.1. in~\cite{MR2180464} after the definition of $\eta$ has been changed to $\eta(v) = \sup \{ \| V \|_{H^1} + \| V \|_{L^1}, \; V \in \mathcal{V}(v_n) \}$.

We apply this proposition to the minimizing sequence $(\rho_n)$, implying first that
$$
\mathcal{H}(\rho_n) - \sum_{j=1}^J \mathcal{H}(V^j) - \frac{1}{8} \| \partial_x \rho_n^J \|_{L^2}^2 + \frac{1}{p} \| \rho_n^J \|_{L^{p/2}}^{p/2} \to 0 \quad \mbox{as $n \to \infty$}
$$
Letting first $n \to \infty$, and then $J \to \infty$, the above proposition implies that
\begin{align*}
\liminf_{n \to \infty} \mathcal{H}(\rho_n) \geq \sum_{j=1}^\infty \mathcal{H} (V^j).
\end{align*}
Denoting $m_j = \mathcal{M}(V^j)$, the above proposition gives that $\sum_j m_j \leq 1$. Using the definition of $I_m$, its scaling, its negativity, and convexity, the last inequality implies that
\begin{align*}
I_1 = \liminf_{n \to \infty} \mathcal{H}(\rho_n) \geq \sum_{j=1}^\infty \mathcal{H} (V^j) \geq I_1 \sum_{j=1}^\infty m_j^{\frac{p+4}{8-p}} \geq I_1 \left( \sum m_j \right)^{\frac{p+4}{8-p}} \geq I_1.
\end{align*}
All the inequality signs above have to be equality signs; but this is only possible if (up to relabeling) $m_1=1$ while $m_j = 0$ for $j \geq 2$. Then $\rho = V_1$ is the desired minimizer!

Since $\rho \in H^1$, it is continuous. Consider an interval $I$ where it is positive, choose $\psi \in \mathcal{C}_0^\infty(I)$, and let $\widetilde{\rho}_\epsilon = \frac{1}{1+\epsilon \int \psi} (\rho + \epsilon \psi)$. Since $\mathcal{M}(\widetilde{\rho}_\epsilon) \geq \mathcal{M}(\rho)$ and $\mathcal{H}(\widetilde{\rho}_\epsilon) \geq \mathcal{H}(\rho)$, we obtain that $\rho$ satisfies
$$
\rho'' + 2 \rho^{\frac{p}{2} - 1} = \lambda,
$$
for a Lagrange multiplier $\lambda$ which might depend on $I$. Coming back to $\phi$, this implies that
$$
\phi = \sum_i \Phi_{B_i,c_i}(x-a_i),
$$
where $\epsilon_i = \pm 1$, $B_i >0$, $c_i \in \mathbb{R}$, and the $a_i \in \mathbb{R}$ are such that the supports of the $\Phi_{B_i,c_i}(\cdot-a_i)$ are disjoint. Letting $m_i = M (\Phi_{B_i,c_i}(x-a_i))$, we have $\sum m_i = 1$. But then
$$
I_1 = H(\phi) = \sum_i H(\Phi_{B_i,c_i}(x-a_i)) \geq I_1 \sum_{i=1}^\infty m_i^{\frac{p+4}{8-p}} \geq I_1 \left( \sum_{i=1}^\infty m_i \right)^{\frac{p+4}{8-p}} = I_1.
$$
Once again, this implies that all but one of the $m_i$ are zero, or in other words,
$$
\phi = \pm \Phi_{B,c}.
$$
If $p=4$, we learn from Proposition~\ref{propminH} that the minimizing compacton is such that $B=0$.
\end{proof}

\subsection{Variational analysis for $p \geq 8$}

When $p \geq 8$, we must deal with the lack of compactness that occurs when traditionally trying to minimize energy with respect to fixed mass.  One approach to this is to use the so-called Weinstein functional as introduced in \cite{weinstein1983nonlinear}.   

For simplicity, take $p=8$ to start.  In such a case, the optimization procedure is to maximize a functional built from the Gagliardo-Nirenberg inequality \eqref{L1L2GNineq}.  The functional is of the form
\begin{equation}
W[u] = \frac{\| u \|_{L^8}^8}{\| u \|_{L^2}^\alpha 
\| u \partial_x u \|_{L^2}^\beta},
\end{equation}
with $\alpha = 8$, $\beta = 4$.  The existence proof works nearly identically to the existence arguments in \cite{weinstein1983nonlinear} via scaling arguments with the appropriately modified scalings to deal with the degenerate $H^1$-type norm. For $p > 8$, a similar strategy will work provided one correctly modifies the powers in Gagliardo-Nirenberg inequality from \eqref{L1L2GNineq} to give
\begin{equation}
\label{L2L2GNineq}
\| u \|_{L^{p}}^p \lesssim \| u \|_{L^2}^{\alpha} \| u u_x \|_{L^2}^{\beta}
\end{equation}
for $\alpha = \frac{(p + 4)}{3}$ and $\beta = \frac{(p-2)}3$.  We remark here that in a similar fashion the best constant in the Galgliardo-Nirenberg inequality \eqref{L1L2GNineq} is given by a function of the $L^2$ norm of the ground state solution for each $p$.

One could also propose an alternative constrained minimization given by finding the minimizer of
\begin{equation}
F_\lambda (u) = \| u \partial_x u \|_{L^2}^2 + \lambda \| u \|_{L^2}^2
\end{equation}
such that 
\begin{equation}
\| u \|_{L^p}^p = \beta.
\end{equation}
Solutions of this minimization can be seen to solve the correct ODE equation after a suitable scaling argument is applied.  For a fairly general treatment of these strategies in a general framework for semilinear operators, see for instance \cite{CMMT}.

\section{Solitary waves for~\eqref{degNLS}}

\label{sectiondegenerateNLS}

\subsection{The ODE}
\label{ODENLS}
By definition, traveling waves are solutions of the form
$$
u = Q ( x - vt ) e^{-ict},
$$
with $v,c \in \mathbb{R}$. They satisfy the equation
\begin{equation}
\label{ODEQ}
- c Q + iv Q' + \overline{Q} (QQ')' + |Q|^{p-2} Q = 0.
\end{equation}
For $v=0$, a first class of solutions is obviously given by $Q = \Phi_{B,c}$, defined in the previous section. Our aim, however, is to completely describe finite mass solutions. This is achieved in the following theorem

\begin{thm} Assume $p>2$. For $B>0$, or $B=0$ and $c>0$, define $\theta_{B,c}$ by
$$
\left\{
\begin{array}{l}
\theta_{B,c}(0) = 0 \\
\displaystyle \theta_{B,c}' = - \frac{1}{2 \Phi_{B,c}^2} \quad \mbox{if $x \in (-x_{B,c},x_{B,c})$}
\end{array}
\right.
$$
so that, if $x \to 0$, $\theta_{B,c}(x-x_{B,c}) \sim
\left\{ \begin{array}{ll} - (cx)^{-1} & \mbox{for $B=0,c>0$} \\ (2\sqrt{2B})^{-1} \log|x| & \mbox{for $B>0$} \end{array}. \right.$ Define furthermore
$$
Q_{B,c}^v (x) = 
\left\{
\begin{array}{ll}
\displaystyle \Phi_{B,c} (x) e^{i v \theta_{B,c}(x)} & \mbox{if $x \in (-x_{B,c},x_{B,c})$} \\
0 & \mbox{if $x \notin (-x_{B,c},x_{B,c})$}
\end{array}
\right.
$$
Then $Q^v_{B,c}$ is a finite mass solution of~\eqref{ODEQ} (in the sense of distributions) with
\begin{equation}
\label{MKH}
M(Q^{B,c}_v) = M(\Phi_{B,c}), \quad K(Q^{B,c}_v) = v x_{B,c} \quad \mbox{and} \quad H(Q^{B,c}_v) = H(\Phi_{B,c}) + \frac{v^2}{4}x_{B,c}.
\end{equation}
Furthermore, any finite mass solution of~\eqref{ODEQ} is of the form
$$
Q = \sum_{i=1}^N \epsilon_i Q_{B_i,c}^v (x-x_i)
$$
where $N \in \mathbb{N}$, $\epsilon_i \in \mathbb{C}$ with $|\epsilon_i| =1$, either $B_i >0$ or $c>0$, and $x_i \in \mathbb{R}$ are such that the supports of the $Q_{B_i,c}^v$ are disjoint.
\end{thm}

\begin{proof} A computation shows that the $Q^v_{B,c}$ are indeed solutions in the sense of distributions. We now prove that they are the only ones.

Multiplying~\eqref{ODEQ} by $\overline{Q}$ and taking the imaginary part, or multiplying the equation by $\overline{Q'}$ and taking the real part leads to the identities
\begin{align*}
& \left[ \frac{v}{2} |Q|^2 + \mathfrak{Im} (|Q|^2 \overline{Q} Q') \right]' = 0 \\
& \left[ -c |Q|^2 + |Q|^2 |Q'|^2 + \frac{2}{p} |Q|^p \right]' = 0.
\end{align*}
Therefore, there exist real constants $\eta$ and $\kappa$ such that, as long as $Q$ does not vanish,
\begin{subequations}
\begin{align}
\label{penguin1} & \frac{v}{2} |Q|^2 + \mathfrak{Im} (|Q|^2 \overline{Q} Q') = \eta \\
\label{penguin2} & -c |Q|^2 + |Q|^2 |Q'|^2 + \frac{2}{p} |Q|^p = \kappa.
\end{align}
\end{subequations}
If the traveling wave is to have finite mass, then necessarily $Q(s) \to 0$ as $s$ approaches some $s_0$ (which might be infinite). Letting $s \to s_0$, we learn from~\eqref{penguin2} that $QQ'$ is bounded. Turning to~\eqref{penguin1}, we learn that
$$
\eta = 0.
$$
To pursue the discussion, we split it into two cases

\bigskip

\noindent \underline{Case 1: $v=0$.} If $v=0$, equation~\eqref{penguin1} implies that $\mathfrak{Im} \overline Q Q' = 0$. This in turn means that $Q$ has a constant phase, so that we can assume, using the symmetries of the equation, that $Q$ is real valued. It solves
$$
c Q = Q(QQ')' + Q^{p-1},
$$
which brings us back to the previous section.

\bigskip

\noindent \underline{Case 2: $c\neq 0$.} If $v \neq 0$, equation~\eqref{penguin1} implies that $\mathfrak{Im} \overline Q Q' = -\frac{v}{2}$, or in other words, adopting the polar form $Q = \psi e^{i\theta}$, that
$$
\theta' = - \frac{v}{2 \psi^2}.
$$
Plugging the ansatz $Q = \psi e^{i \theta}$ into the ODE~\eqref{ODEQ}, leads to
$$
c \psi = \psi (\psi \psi')' + \psi^{p-1},
$$
so that we are back to Case 1.
\end{proof}

\subsection{Variational properties} It is clear that the results in Section~\ref{subsectionvariationaldKdV} extend to the case of complex-valued functions: namely, the complex-valued minimizers of $H(u)$ subject to $M(u)$ constant coincide with the real-valued minimizers.

In analogy with the semilinear case, it would be natural to expect that the $\Phi_{B,c}^v$ appear as the minimizers of
$$
\min H(u) \quad \mbox{subject to $M(u) = M_0$ and $K(u) = K_0$};
$$
but we will see it is not the case. Adopting the polar decomposition $u = \sqrt{\rho} e^{i\theta}$, this becomes.
$$
\min \frac{1}{8} \int |\partial_x \rho|^2 \,dx + \frac{1}{2} \int \rho^2 |\partial_x \theta|^2 \,dx- \frac{1}{p} \int \rho^{p/2}\,dx \quad \mbox{suject to} 
\quad 
\left\{ 
\begin{array}{l} 
\int \rho\,dx =M_0 \\ \int \rho \partial_x \theta = K_0
\end{array}
\right.
$$
A non-compact minimizing sequence can be constructed as follows: consider $\chi$ in $\mathcal{C}^\infty_0$ be radial, with support $(-1,1)$, $\int \chi = 1$, and let $\chi_R = \chi\left( \frac{\cdot}{R} \right)$. Next, let $\zeta_{\epsilon, R}$ solve $\partial_x \zeta_{\epsilon,R} = \frac{K_0}{2 \epsilon^2 R \chi_R}$. Finally, let $\phi$ be a minimizer of $H$ subject to $M=M_0$ and define
$$
u = \phi + \epsilon \sqrt{\chi_R} e^{i \zeta_{\epsilon,R}} [\cdot - 10R],
$$
so that most of the mass and the energy lies in the first summand, while all the momentum is contained in the second, non compact, summand. To be more precise, a small computation reveals that
\begin{align*}
& M(u) = M(\phi) + \epsilon^2 R = M_0 + \epsilon^2 R \\
& K(u) = K_0 \\
& H(u) = H(\phi) + \frac{\epsilon^4}{R} \int |\partial_x \chi|^2 - \frac{\epsilon^p R}{p} \int |\chi|^{p/2} + \frac{K_0^2}{8R}.
\end{align*}
Letting $R \to \infty$ and choosing for instance $\epsilon = R^{-100}$ gives 
a minimizing sequence such that
$$
M(u) \to M_0, \quad K(u) \to K_0, \quad H(u) \to \min_{M=M_0} H.
$$
Due to~\eqref{MKH}, this example shows that $Q_{B,c}^v$ cannot be minimizers; it also illustrates the basic lack of compactness which explains the absence of minimizers.

\subsection{Hydrodynamic formulation} The existence of traveling waves for~\eqref{degNLS} can be read off most easily after switching to hydrodynamic coordinates: taking the Madelung transform $u(t,x) = \sqrt{\rho (x,t)} e^{i \theta(x,t)}$ leads to the equation
\begin{align*}
& \rho_t + 8 \partial_x ( \rho^2 \theta_x) = 0, \\
& \theta_t + 8 \rho \theta_x^2 = 4 \rho_{xx} + 3 \rho.
\end{align*}
Defining $v = 8 \rho \theta_x$ as the designated flow velocity, the equation becomes
\begin{align}
\label{hydro:NLS1}
& \rho_t +  \partial_x ( \rho v) = 0, \\
& v_t + 3 v \partial_x v =  \rho(  \rho_{xxx} + 2 \rho_x). \notag
\end{align}
Traveling waves now correspond to solutions of
$$
\left\{
\begin{array}{l}
v = \operatorname{cst} \\
\rho_{xxx} + 2 \rho_x = 0,
\end{array}
\right.
$$
which provides an alternative (and, in some respects, simpler) proof of the results of Subsection~\ref{ODENLS}.

\bigskip

This can be contrasted with a related model derived by John Hunter \cite[Section $4.1$]{hunter1995asymptotic} that arises as an asymptotic equation for a two-wave system in
a compressible gas dynamics put forward by Majda-Rosales-Schonbek \cite{majda1988canonical}.  The degenerate NLS model is given by
\begin{equation}\label{eq:nls2m}
- iu_t= \partial_x (|u|^2 u_x) , \ \ v: \R \to \C.
\end{equation}
This equation was introduced to the authors during a talk of John Hunter and is now being studied in significant detail by Hunter and graduate student Evan Smothers \cite{HunterPrivate}.

Naively taking the Madelung transformation of \eqref{eq:nls2m}, $v = \sqrt{\rho (x,t)} e^{i \theta(x,t)}$, for this equation, we arrive at
\begin{align*}
& \rho_t +  \partial_x ( \rho u) = 0, \\
& u_t + 2 u \partial_x u + \frac{u^2 (\log \rho)_x}{2} =  \rho \left( \frac12 \frac{\rho_{x}^2}{\rho} +  \rho_{xx} \right)_x
\end{align*}
for $u = 2 \rho \theta_x$.  The right hand side of the $u$ equation no longer quite so clearly supports coherent structures, but might be ideal for the study of shock-like solutions. It is unclear however whether the variational approach which we used for~\eqref{degNLS} will yield useful results. 

\begin{rem}
One might also from Euler systems of this type attempt to derive a weakly dispersive KdV limit as in the study of dispersive shock waves a la Whitham theory.  See for instance the review article of El-Hoefer-Shearer \cite{el2017dispersive}.  
\end{rem}

\section{Linear Stability for Compactons of~\eqref{degKdV} when $p=4$}
\label{sectionlinearstability}

In this section, we study the operator stemming from linearizing \eqref{degKdV} with $p=4$ about solutions
\[
\phi = \phi_{B,c},
\]  
both when $B = 0, c >0$ and $B>0, c = 0, c < 0, c > 0$.  When $p \neq 4$, but $2<p<8$, a similar analysis should follow.  However as $p=4$ is both the power in the original derivation and the most interesting algebraically, we focus on it primarily here. For simplicity of exposition, we will treat the following specific cases:
\begin{enumerate}
\item  $B=0$, $c=1$;
\item  $B = \frac14$: 
\begin{enumerate}
\item $c = 1$;
\item $c = 0$;
\item $c = -1$,
\end{enumerate}
\end{enumerate}
though other values of $B$ and $c$ appropriately related follow with small modifications.

We assume that our initial data is of the form
\[
u_0 = \phi + v_0,
\]
where the perturbation \(v_0\) is sufficiently small, smooth and \(\supp v_0\subseteq \bar I = [-x_{B,c},x_{B,c}]\). For times \(t>0\) we assume that our solution may be written in the form
\[
u(t,x) = \phi\left(x - c t \right) + v\left(t,x - c t\right).
\]
In the moving frame, keeping only the linear terms in $v$ we obtain the equation
\begin{equation}\label{Basic}
\begin{cases} 
v_t = \left( \mathcal{L}_{\phi} v  \right)_x,\vspace{0.1cm}\\
v(0) = v_0,\vspace{0.1cm}\\
v\big|_{\partial I} = 0,\vspace{0.1cm}\\
\mc L_{\phi} v \big|_{x=x_{B,c}}= 0,
\end{cases}
\end{equation}
where linearized operator,
\begin{equation}
\label{Lu}
\mc L_{\phi} = - \phi (\partial_x^2 + 2)\phi,
\end{equation}
can be seen as a singular Sturm-Liouville operator and written in the form
$$
\mathcal{L}_{\phi} = -\phi^2 \partial_x^2 - 2 \phi \phi_x \partial_x - (\phi\phi_{xx} + 2\phi^2).
$$

\begin{rem} The additional boundary condition \( \mc L_{\phi} v \big|_{x=x_{B,c}}= 0\) at the right endpoint is natural from the point of view of KdV equations on bounded intervals, see for example~\cite{MR3582261,MR1998942,MR2254610} and references therein.
\end{rem}

In order to better characterize the behavior of the eigenfunctions of \(\mc L_\phi\) we recall the Sturm comparison and oscillation theorems (see for example \cite[Theorems~9.39, 9.40]{MR3243083}):
\medskip
\begin{thm}[Sturm comparison]
Let \(\lambda_1<\lambda_2\) and \(q_1,q_2\) be solutions to the ODE
\[
\mc L_\phi q_j = \lambda_j q_j,
\]
on some open interval \((a,b)\subseteq I\). Suppose that at each endpoint \(a,b\) either \(W[q_1,q_2] = 0\) or if the endpoint lies in the interior of the interval \(I\) we have \(q_1 = 0\). Then the function \(q_2\) has a zero in the interval \((a,b)\).
\end{thm}
\medskip

\begin{rem}
From standard ODE theory \(q_1,q_2\) are smooth on the open interval \(I\) and hence the Wronskian is well defined on \(I\). If \(a = - x_{B,c}\) then we define
\[
W[q_1,q_2](a) = \lim\limits_{x\downarrow a}W[q_1,q_2](x),
\]
provided such a limit exists. A similar definition holds when \(b = x_{B,c}\). A consequence of \(\mc L_\phi\) being limit point is that for all eigenfunctions \(q_1,q_2\) of \(\mc L_\phi\) this limit exists and is equal to zero,
\[
W[q_1,q_2](\pm x_{B,c}) = 0.
\]
\end{rem}

\medskip
\begin{thm}[Sturm oscillation]
If \(\mc L_\phi\) has eigenvalues \(\lambda_0<\lambda_1<\dots\) with corresponding eigenfunctions \(q_0,q_1,\dots\) then \(q_j\) has exactly \(j\) zeros in the interval \(I\).
\end{thm}
\medskip

We consider \(\mc L_{\phi}\) to be a symmetric unbounded operator on \(L^2(I)\) with domain \(C^\infty_0(I)\). For \(w\in C^\infty_0(I)\) we may integrate by parts to obtain
\[
\<\mc L_{\phi} w,w\> = \|(\phi w)_x\|_{L^2}^2 - 2\|\phi w\|_{L^2}^2 \geq - 2\|\phi\|_{L^\infty}^2\|w\|_{L^2}^2.
\]
We may then associate \(\mc L_{\phi}\) with the corresponding Friedrichs extension.

\subsection{The case $B=0$, $c=1$}
We first consider the setting $\phi = \phi_{0,1}$, where we recall the explicit expression
\[
\phi_{0,1}(x) = \sqrt 2\cos(\frac x{\sqrt 2}),\qquad x\in I = (-\frac\pi{\sqrt2},\frac\pi{\sqrt 2}).
\]
We then have the following properties of the operator \(\mc L_{\phi_{0,1}}\):
\medskip
\begin{prop}\label{prop:phi01}
The operator \(\mc L_{\phi_{0,1}}\) is limit point and satisfies the following:
\begin{enumerate}
\item The ground state energy is \(\lambda_0 = -2\) with positive ground state \(\phi\).
\item The first excited eigenvalue is \(\lambda_1 = 0\) with corresponding eigenfunction \(\phi_x\).
\item The operator has continuous spectrum \([\frac14,\infty)\).
\end{enumerate}
\end{prop}

\medskip
First, let us observe that $\mathcal{L}_\phi$ is limit point. Indeed, $\mathcal{L}_\phi u =0$  is equivalent, upon setting $w = \phi u$, to $(\partial_x^2 + 2)w = 0$. Therefore, the general solution reads $u = \frac{a \cos(\sqrt 2 x) + b \sin(\sqrt{2} x)}{\phi}$; for $a = 1$, $b=0$, this solution is not square integrable at either endpoint so by the Weyl alternative the operator \(\mc L_\phi\) is limit point at \(x = \pm \frac\pi{\sqrt 2}\) and the Friedrichs extension is the unique self-adjoint extension of \(\mc L_\phi\) (see for example \cite[Theorems~9.6,~9.9]{MR3243083}).

Under the (invertible) transformation
\begin{equation}
t(x) = -\int_{0}^x \phi^{-1} (s) ds,
\end{equation}
we can translate the operator \(\mc L_\phi\) to standard elliptic problem on all of $\mathbb{R}$. To see this, we will reformulate \eqref{Lu} as a simple $b$-operator on $I$.  This is the strategy of the standard $b$-calculus as developed by Melrose and many others, see e.g. \cite{bcalc1,bcalc2,bcalc3,bcalc5,bcalc4}.

We easily observe that we have
\[
\partial_t = \phi \partial_x,
\]
and as a result
\[
\phi \partial_x^2 \phi = \phi^2 \partial_x^2 + 2 \phi \phi_x \partial_x + \phi \phi_{xx}  =  \partial_t^2 + \phi_x (x(t)) \partial_t + V (t),
\]
where
\[
V (t) = \phi (x(t) )  \phi_{xx} (x(t)) = - \frac12 \phi(x(t))^2.
\]
From the asymptotics of $\phi$ near $\pm \frac\pi{\sqrt 2}$, we recognize that $|V (t)| \sim e^{- C|t|}$ as $t \to \pm \infty$. Conjugating by the integrating factor
\[
g(t) = e^{-\frac12 \int_0^t \phi_x (x(s))\ ds },
\]
we arrive at the simple elliptic operator
\begin{equation}
\label{linsmoothcomp_bcalc}
\mathcal{L}_b = -\partial_t^2 + \frac14 + \frac{15}4V(t),
\end{equation}
which must then have the same spectrum as the operator \(\mc L_\phi\).

The operator $\mathcal{L}_{b}$ is thus seen to be a relatively compact perturbation of $-\partial_t^2 + \frac14 $ and hence by Weyl's Theorem has continuous spectrum $\sigma_c ( \mathcal{L}_{b} ) = [\frac14 , \infty)$. With regards to the discrete spectrum, note that $\mathcal{L}_\phi \phi = -2 \phi $ and $\phi (x(t))$ is a nice $L^2$ function in the $t$ variables, in fact, it is exponentially decaying. Plugging in $w = g^{-1} \phi$, we observe that $\mathcal{L}_b $ has a negative eigenvalue at $\lambda = - 2$. We also have $\mathcal{L}_\phi \phi_x = 0 $, which turns into an eigenvalue for $\mathcal{L}_{b}$ also at $\lambda = 0$ given that $g^{-1} \phi_x$ is exponentially decaying as $t \to \pm \infty$.  By Sturm oscillation theory for elliptic operators, we can see that there are no eigenvalues between $\lambda = -2$ and $\lambda  = 0$.  

This is sufficient to set up elliptic estimates for $\mathcal{L}_b$ and do the finite time modulation theory proposed in this work. For future work on long and/or global time scales, it is important to understand dispersive estimates for the linearized operator.  In such a case, we will need to potentially rule out a resonance at the endpoint of the continuous spectrum $\lambda = \frac14$, see \cite{schlag2007dispersive}.

\subsubsection{Spectral Theory without the Nonlinear Transformation}

It is possible to determine the spectrum without changing coordinates: first notice that, due to the ODE satisfied by $\phi$,
$$
\mathcal{L}_\phi \phi = -2 \phi \quad \mbox{and} \quad \mathcal{L}_\phi \phi_x = 0.
$$
Since $\phi_x$ vanishes once on $I$, we deduce, by the Sturm oscillation theorem~\cite{MR3243083}, that $\phi_x$ is the first excited mode.

To determine the behavior at the endpoints of I of a solution of $\mathcal{L}_\phi u = \lambda u$, which can also be written
$$
\left[ - \phi^2 \partial_x^2 - 2 \phi \phi_x \partial_x - \frac{3}{2} \phi^2 \right] u = \lambda u.
$$
We now apply Frobenius' method~\cite{MR2961944}. Since both endpoints are symmetrical, it suffices to consider the left endpoint $-\frac{\pi}{\sqrt 2}$. 
Switching variable to $z = x + \frac{\pi}{\sqrt 2}$, observe that $\phi(z - \frac{\pi}{\sqrt 2}) \sim z$. Only the top orders of the expansion of the coefficients matters for Frobenius' method, so that we can consider the equation
$$
- z^2 \partial_z^2 u - 2 z \partial_z u = \lambda u.
$$
The indicial equation (obtained by plugging $u(z) = z^\alpha$ in the above) reads $\alpha^2 + \alpha + \lambda = 0$, leading to the characteristic exponents
$$
\alpha_{\pm} = \frac{-1 \pm \sqrt{1 - 4 \lambda} }{2}.
$$
Frobenius' method gives a basis of solutions behaving $\sim z^{\alpha_{\pm}}$ as $z \to 0$. Observe that this leads to an infinite number of oscillations for $\lambda > \frac{1}{4}$. Thus, the continuous spectrum is $[\frac{1}{4},\infty)$, see \cite[Page 220]{weidmann2006spectral}.

\subsubsection{The solution operator for $\mc L_\phi q = f$ in general for $B=0$, $c=1$}

A calculation shows that linearly independent solutions to $\mc L_\phi q = 0$ now correspond to $q_1 = \phi_x$ and $q_2 = \phi - \phi^{-1}$.  The function $q_1 \in L^2$ since $\phi_x \to \mp 1$ as $x \to \pm x_{1,0}$.  The function $q_2 \notin L^2$ in either direction.  Note, the modified Wronskian satisfies 
\begin{equation}
\label{Wron_B0_c1}
W_\phi (q_1,q_2) = (\phi q_1)_x (\phi q_2) - (\phi q_2)_x (\phi q_1) = \frac12.
\end{equation}
If \(f\in C^\infty_0(I)\), the general solution \(w \in L^2(I)\) is given by
\begin{equation}\label{VarPar_B0_c1}
\begin{aligned}
\mc L_\phi^{-1}f(x) &= q_1(x)\left(2 \int_{-x_{0,1}}^x f(y) q_2 (y)\,dy - 2 \int_x^{x_{0,1}}f(y) q_2 (y)\,dy   \right)\\
&\quad + q_2 (x)\left( 2 \int_{-x_{0,1}}^xf(y)q_1(y)\,dy - 2 \int_x^{x_{0,1}}f(y)q_1(y)\,dy \right) \\ 
&\quad + C \phi_x,
\end{aligned}
\end{equation}
where \(C\in \R\).

To construct a semigroup for \( e^{t \partial_x \mc L_\phi} \) below for \(B = 0\), \(c=1\), we need that \(\mc L_\phi^{-1}\colon L^2_\perp\rightarrow L^2_\perp\) is a reasonable operator where
\[
L^2_\perp = \{w\in L^2:\<w,\phi\> = 0 = \<w,\phi_x\>\}.
\]
We thus collect the following result:

\medskip
\begin{lem}\label{lem:HS2_B0_c1}
The operator \(\mc L_\phi^{-1}\colon L^2_\perp\rightarrow L^2_\perp\) is continuous where
\[
L^2_\perp = \{w\in L^2:\<w,\phi\> = 0 = \<w,\phi_x\>\}.
\]
\end{lem}

\begin{proof}
This follows immediately from the spectral theory above.  In particular, $\mc L_\phi$ is bounded away from $0$ from below on $L^2_\perp$, meaning that $\mc L_\phi^{-1}$ is a bounded operator on this space and hence continuous.  

One can also prove this using directly the form of \eqref{VarPar_B0_c1}. 
\end{proof}
\medskip

\subsection{The case $B=\frac14$}~
We now turn to the case \(\phi = \phi_{\frac14,c}\) and recall the explicit expression,
\[
\phi_{\frac14,c} = \sqrt{c + \sqrt{1 + c^2}\cos(\sqrt 2x)},\qquad x\in I = (-x_{\frac14,c},x_{\frac14,c}), \qquad \tan(\sqrt 2 x_{\frac14,c}) = -\frac1c.
\]
We summarize the properties of the operator \(\mc L_{\phi_{\frac14,c}}\) as follows:
\medskip
\begin{prop}\label{prop:LinearOpProps}
For \(c = \pm 1,0\) the operator \(\mc L_\phi\) satisfies the following:
\begin{enumerate}
\item The operator \(\mc L_\phi\) limit point at \(x = \pm x_{\frac14,c}\)
\item The operator has discrete spectrum and there exists an orthogonal basis of \(L^2(I)\) consisting of simple eigenfunctions of \(\mc L_\phi\).
\item The ground state is given by \(\phi\) with corresponding ground state energy \(\lambda_0 = - 2c\).
\item The first eigenvalue \(\lambda_1>\lambda_0\) is positive.
\end{enumerate}
\end{prop}
\medskip

\subsubsection{The operator \(\mc L_\phi\) is limit point}
To show the operator is limit point we observe as before that the general solution to the homogeneous equation
\(
\mc L_\phi u = 0
\)
is given by \(u = \frac{a\cos(\sqrt 2x) + b\sin(\sqrt 2x)}{\phi(x)}\).

When \(c = \pm 1\) we may take \(a = \mp \frac1{\sqrt 2\sqrt{1 + c^2}}\), \(b =  - \frac c{\sqrt 2\sqrt{1 + x^2}}\), to obtain two linearly independent solutions to \(\mc L_\phi q_\pm = 0\) given by
\begin{equation}\label{qpm}
q_\pm(x) = \frac{\sin(\sqrt 2(x \mp x_{\frac14,c}))}{\sqrt 2\phi(x)},
\end{equation}
that satisfy the boundary conditions
\[
q_\pm(\pm x_{\frac14,c}) = 0,\qquad (\phi q_\pm)_x(\pm x_{\frac14,c}) = 1.
\]
As the length of the interval \(|I| \neq \frac{\pi}{2\sqrt 2}\) we have
\begin{alignat*}{3}
q_- &\in L^2((-x_{\frac14,c},0)),&&\qquad q_-&&\not\in L^2((0,x_{\frac14,c})),\\
q_+&\not\in L^2((-x_{\frac14,c},0)),&&\qquad q_+&&\in L^2((0,x_{\frac14,c})),
\end{alignat*}
so the operator is limit point.

When \(c = 0\) we have \(q_- = - q_+ = \frac1{\sqrt 2}\phi\). However, we may construct a second solution to the homogeneous ODE by taking
\[
q_* = \frac{\sin(\sqrt 2 x)}{\sqrt 2\phi(x)},
\]
and note that the operator is still limit point as \(q_*\not\in L^2(I)\).

\subsubsection{The operator \(\mc L_\phi^{-1}\) is Hilbert-Schmidt}
We now consider solutions to the ODE
\begin{equation}\label{InhomoODE}
\mc L_\phi w = f.
\end{equation}
If \(c = \pm 1\) and \(f\in C^\infty_0(I)\), the unique solution \(w \in L^2(I)\) is given by
\begin{equation}\label{VarParam}
\mc L_\phi^{-1}f(x) = C_c q_+(x)\int_{-x_{\frac14,c}}^x q_-(y)f(y)\,dy + C_c q_-(x) \int_x^{x_{\frac14,c}}q_+(y)f(y)\,dy,
\end{equation}
where the constant \(C_c>0\) is given by the (modified) Wronskian of the functions \(q_\pm\),
\[
\frac1{C_c} = W[q_-,q_+] = (\phi q_-)(\phi q_+)_x - (\phi q_+)(\phi q_-)_x = \mp\frac{\sqrt 2 c}{1 + c^2},\qquad c = \pm1.
\]
%
Using the formula \eqref{VarParam}, we may then show that operator \(\mc L_\phi^{-1}\) extends to a compact operator on \(L^2(I)\):
\medskip
\begin{lem}\label{lem:HS}
If \(c = \pm 1\) the operator \(\mc L_\phi^{-1}\) extends to a Hilbert-Schmidt operator on \(L^2(I)\). In particular, the operator \(\mc L_\phi\) has discrete spectrum and there exists a basis of \(L^2(I)\) consisting of simple eigenfunctions of \(\mc L_\phi\).
\end{lem}
\begin{proof}
The kernel of the operator \(\mc L_\phi^{-1}\) is given by
\[
K(x,y) = C_c q_+(x)q_-(y)\mbf1_{(-x_{\frac14,c},x)}(y) + C_c q_-(x)q_+(y)\mbf1_{(x,x_{\frac14,c})}(y).
\]
Taking \(\bd(x) = \dist(x,\partial I)\) we may use the fact that \(\phi(x)\sim \bd^{\frac12}\) near \(x = \pm x_{\frac14,c}\) and the explicit expression \eqref{qpm} to obtain the bounds
\[
|q_\pm(x)|^2\sim\begin{cases}\bd(x),&\quad \pm x>0\\\bd(x)^{-1},&\quad \pm x<0.\end{cases}
\]
In particular,
\[
\int_{-x_{\frac14,c}}^x |q_-(y)|^2\,dy \sim \begin{cases}\bd(x)^2,&\ x<0\\ 1 + |\log\bd(x)| ,&\ x>0\end{cases},\qquad \int^{x_{\frac14,c}}_x |q_+(y)|^2 \,dy \sim \begin{cases}\bd(x)^2,&\ x>0\\1 + |\log\bd(x)|,&\ x<0\end{cases},
\]
and hence 
\[
\int_{-x_{\frac14,c}}^{x_{\frac14,c}}|K(x,y)|^2\,dy\lesssim 1.
\]
Integrating this in \(x\) we see that the operator \(\mc L_\phi^{-1}\) is Hilbert-Schmidt and hence compact. The fact that the eigenvalues of \(\mc L_\phi\) are simple follows from the fact that \(\mc L_\phi\) is limit point at \(x = \pm x_{\frac14,c}\).
\end{proof}
\medskip

When \(c = 0\) we require an orthogonality condition to obtain a unique solution to the ODE \eqref{InhomoODE}. Thus we define the space
\[
L^2_\perp(I) = \{w\in L^2(I):\<w,\phi\> = 0\}.
\]
For \(f\in C^\infty_0(I)\) satisfying \(\<f,\phi\> = 0\) we take \(w = \mc L_\phi^{-1}f\) to be the unique solution \(w\in L^2_\perp(I)\) to the equation \eqref{InhomoODE} given by
\begin{equation}\label{VarParamc0}
\begin{aligned}
\mc L_\phi^{-1}f(x) &= q_*(x)\left(\frac12\int_{-x_{\frac14,0}}^xf(y)\phi(y)\,dy - \frac12\int_x^{x_{\frac14,0}}f(y)\phi(y)\,dy \right)\\
&\quad + \phi(x)\left(\frac12\int_{-x_{\frac14,0}}^xf(y)q_*(y)\,dy - \frac12\int_x^{x_{\frac14,0}}f(y)q_*(y)\,dy - \frac1{2\sqrt 2} \int_{-x_{\frac14,0}}^{x_{\frac14,0}}f(y)\,dy\right).
\end{aligned}
\end{equation}
An essentially identical proof to Lemma~\ref{lem:HS} yields the following:
\medskip
\begin{lem}\label{lem:HS2}
If \(c = 0\) the operator \(\mc L_\phi^{-1}\) extends to a Hilbert-Schmidt operator on the space \(L^2_\perp(I)\). In particular, the operator \(\mc L_\phi\) has discrete spectrum and there exists a basis of \(L^2_\perp(I)\) consisting of simple eigenfunctions of \(\mc L_\phi\).
\end{lem}
\medskip

\subsubsection{The ground state and first harmonic}
Recalling that \(\mc L_\phi = - 2c\phi\) has no zeros in the interval \(I\) we see that this is the ground state. In particular, for \(c = -1, 0\) we must have that the first harmonic has positive eigenvalue \(\lambda_1>0\). In the case \(c= 1\) we have the following lemma:
\medskip
\begin{lem}\label{lem:Coercive}
When \(c=1\) the first eigenvalue \(\lambda_1>\lambda_0\) is positive.
\end{lem}
\begin{proof}
Suppose for a contradiction that \(\lambda_1<0\). Let \(x_*\in I\) be the zero of the corresponding eigenfunction \(q_1\) and, recalling that the length of the interval \(|I| = 2x_{\frac14,1} \in (\frac\pi{\sqrt2},\sqrt2\pi)\), we must either have \(x_*<\frac{\pi}{\sqrt 2} - x_{\frac14,1}\) or \(x_*>x_{\frac14,1} - \frac{\pi}{\sqrt 2}\). Without loss of generality assume that \(x_*<\frac{\pi}{\sqrt 2} - x_{\frac14,1}\). If \(q\) lies in the domain of \(\mc L_\phi\), then we must have that \(\phi q\in H^1_0(I)\)
so recalling the Wronskian is given by
\[
W[q_-,q] = (\phi q_-)(\phi q)_x - (\phi q_-)_x(\phi q),
\]
we see that \(W[q_-,q](-x_{\frac14,1}) = 0\). In particular, taking \(q = q_1\) we may apply the Sturm comparison principle on the interval \((-x_{\frac14,1},x_*)\) to show that \(q_-\) must have a zero in the interval \((-x_{\frac14,1},x_*)\), which is a contradiction. Thus, \(\lambda_1>0\).
\end{proof}
\medskip

\subsection{Energy spaces}
The equation \eqref{Basic} has a formally conserved energy,
\[
E_\phi[w] = \|(\phi w)_x\|_{L^2}^2 - 2\|\phi w\|_{L^2}^2 = \< \mc L_\phi w,w\>.
\]
From Propositions~\ref{prop:phi01},~\ref{prop:LinearOpProps} we see that
\[
\|w\|_{L^2}^2\leq (1 + 2c)\|w\|_{L^2}^2 + E_\phi[w],
\]
and hence we may define a natural energy space \(\X\subset L^2(I)\) associated to the equation \eqref{Basic} given by the completion of \(C^\infty_0(I)\) under the norm
\[
\|w\|_{\X}^2 = (1 + 2c)\|w\|_{L^2}^2 + E_\phi[w].
\]

As the \(L^2\)-norm is not conserved by \eqref{Basic} it is natural to define the subset \(\dot \X\subseteq \X\) where we define
\[
\dot \X = \begin{cases}
\left\{w\in \X:\<w,\phi\> = 0 = \<w,\phi_x\>\right\},&\quad B = 0,c = 1\\
\left\{w\in \X:\<w,\phi\> = 0\right\},&\quad B = \frac14, c = 1,0\\
\X,&\quad B = \frac14, c = -1,
\end{cases}
\]
where we note that the orthogonality condition
\[
\< w,\phi\> = 0
\]
is conserved by the flow of \eqref{Basic} for all choices of \(B,c\) and that the orthogonality condition
\[
\< w,\phi\> = 0 = \<w,\phi_x\>
\]
is conserved when \(B = 0\) due to the fact that \(\phi_{xx} = - \frac12\phi\). A simple consequence of Propositions~\ref{prop:phi01},~\ref{prop:LinearOpProps} is that whenever \(w\in \dot\X\) we have the estimate
\begin{equation}\label{Coercive}
\|w\|_{L^2}^2\lesssim E_\phi[w],
\end{equation}
thus it is natural to define the norm
\[
\|w\|_{\dot \X}^2 = E_\phi[w],
\]
with associated inner product
\[
(f,g)_{\dot\X} = \<(\phi f)_x,(\phi g)_x\> - 2\<\phi f,\phi g\>.
\]

We now consider the operator \(\partial_x\mc L_\phi\) as an unbounded operator on \(\dot \X\) with domain
\[
\dot \X^1 = \left\{w\in \dot \X:\partial_x\mc L_\phi w\in \dot \X,\ \mc L_\phi w\big|_{x = x_{B,c}} = 0 \right\}.
\]
We may endow \(\dot \X^1\) with the norm
\[
\|w\|_{\dot \X^1}^2 = \|w\|_{\dot \X}^2 + \|\partial_x\mc L_\phi w\|_{\dot \X}^2,
\]
and then have the following lemma:
\medskip
\begin{lem}\label{lem:compact}
The embedding \(\dot \X^1\subset \dot \X\) is compact.
\end{lem}
\begin{proof}
If \(w\in \dot \X^1\) then we may use the fact that \(\mc L_\phi w\big|_{x = x_{B,c}} = 0\) and the estimate \eqref{Coercive} to obtain
\[
\|\mc L_\phi w\|_{H^1(I)}\lesssim \|\partial_x\mc L_\phi w\|_{L^2}\lesssim \|\partial_x\mc L_\phi w\|_{\dot \X} \lesssim\|w\|_{\dot\X^1},
\]
where \(H^1(I)\) is the usual Sobolev space on \(I\).

Now let \(\{u^{(j)}\}\subset \dot\X^1\) be a bounded sequence. Then \(\{\mc L_\phi u^{(j)}\}\subset H^1(I)\) is also bounded so as the embedding \(H^1(I)\subset L^2(I)\) is compact, passing to a subsequence there exists some \(v\in L^2\) so that
\[
\mc L_\phi u^{(j)} \rightarrow v\text{ in }L^2(I).
\] 
Next we define \(u = \mc L_\phi^{-1}v\in L^2(I)\) (or \(\in L^2_\perp(I)\) when \(B = 0\) or \(c = 0\)) and by continuity of \(\mc L_\phi^{-1}\),
\[
u^{(j)}\rightarrow u\text{ in }L^2(I).
\]

Finally we observe that if \(f,g\in \dot\X\) and \(\mc L_\phi f\in L^2(I)\),
\[
|(f,g)_{\dot\X}|\leq \|\mc L_\phi f\|_{L^2}\|g\|_{L^2},
\]
where the integration by parts may be justified by observing that whenever \(f\in \dot \X\) we have \(\phi f\in H^1_0(I)\) and hence \(\phi f|_{\partial I} = 0\). Taking \(f = g = u^{(j)} - u\) in this identity we then obtain
\[
u^{(j)}\rightarrow u\text{ in }\dot \X,
\]
as required.
\end{proof}
\medskip

\subsection{The linear semigroup and local well-posedness of a linear equation}
For all \(w\in \dot \X^1\) we have
\begin{equation}\label{NonnegativeGoodness}
(\partial_x\mc L_\phi w,w)_{\dot\X} = \frac12\left(\mc L_\phi w\right)^2\big|_{x = - x_{B,c}}^{x = x_{B,c}}\leq 0,
\end{equation}
and hence \(\partial_x\mc L_\phi\) is dissipative. This then allows us to construct a semigroup \(S(t) = e^{t\partial_x\mc L_\phi}\):
\medskip
\begin{lem}\label{lem:Semigroup}
The operator \(\partial_x\mc L_\phi\) generates a contraction semigroup \(S(t) = e^{t\partial_x\mc L_\phi}\colon \dot\X\rightarrow\dot\X\).
\end{lem}
\begin{proof}
It suffices to show that \(1\not\in \sigma(\partial_x\mc L_\phi)\). By construction the inverse operator \((\partial_x\mc L_\phi)^{-1}\colon \dot\X\rightarrow \dot\X^1\) is well-defined. From Lemma~\ref{lem:compact} the embedding \(\dot \X^1\subset \dot \X\) is compact and hence \((\partial_x\mc L_\phi)^{-1}\colon \dot\X\rightarrow \dot\X\) is a compact operator on \(\dot\X\). From the estimate \eqref{NonnegativeGoodness} there does not exist a non-trivial solution \(w\in \dot\X^1\) to the homogeneous equation
\[
(1 - \partial_x\mc L_\phi)w = 0,
\]
so applying the Fredholm alternative to the operator \(1 - (\partial_x\mc L_\phi)^{-1}\) we see that \(1\not\in \sigma(\partial_x\mc L_\phi)\) as required.
\end{proof}
\medskip

We note that using the semigroup \(S(t)\) we may construct a mild solution \(v\in C([0,T);\dot \X)\) to the linear equation
\begin{equation}\label{Basic2}
\begin{cases} 
v_t = \left( \mathcal{L}_{\phi} v  \right)_x + f,\vspace{0.1cm}\\
v(0) = v_0,\vspace{0.1cm}\\
v\big|_{\partial I} = 0,\vspace{0.1cm}\\
 \mc L_{\phi} v \big|_{x=x_{B,c}}= 0,
\end{cases}
\end{equation}
whenever \(v_0\in \dot \X\) and \(f\in L^1([0,T);\dot \X)\) using the Duhamel formula
\[
v(t) = S(t)v_0 + \int_0^t S(t - s)f(s)\,ds.
\]
Further, the solution \(v\) satisfies the energy estimate
\[
\frac d{dt}\|v\|_{\dot \X}^2 + \Upsilon^2 \leq 2|(f,v)_{\dot \X}|
\]
where
\[
\Upsilon =  \lim\limits_{x\downarrow -x_{B,c}}\mc L_\phi w.
\]
This construction may then be straightforwardly extended to handle initial data in the energy space \(\X\) using a simple modulation argument.

\medskip
\begin{rem}\label{rem:Distribs}
We remark that the solution to \eqref{Basic2} does not quite extend to a solution to the linearized equation on \(\R\) in the sense of distributions. Indeed, if we take \(\mc E\colon \dot\X\rightarrow \mc D'(\R)\) to be extension by zero, then taking \(V = \mc E v\) and \(F = \mc Ef\), as distributions on \(\R\) we obtain
\[
V_t = (\mc L_\phi V)_x + F - \Upsilon \delta_{x = -x_{B,c}}
\]
\end{rem}
\medskip

\subsection{The operators for Linearized \eqref{degNLS}}

When $v=0$, a somewhat tedious calculation reveals that linearizing about $\phi$ in the case of degenerate NLS results in an operator of the form
\begin{equation}
\label{HdegNLS}
\mathcal{H}_0 = \left[  \begin{array}{cc}
0 &    -(\mathcal{L}_{\phi} + 2c) \\
\mathcal{L}_{\phi} & 0
\end{array} \right],
\end{equation}
when acting on a perturbation $w = w_1 + i w_2$, where $\mathcal{L}_{\phi}$ is as above.  Acting on $w, \overline{w}$, we have
\begin{equation}
\label{HdegNLSalt}
\tilde{\mathcal{H}}_0 = -i  \left[  \begin{array}{cc}
\phi \partial_x^2 (\phi \cdot) - \phi_{xx} +2 \phi^2 - c & \partial_x (\phi \phi_x) + \phi^2  \\
-( \partial_x (\phi \phi_x) + \phi^2)  &- (\phi \partial_x^2 (\phi \cdot)- \phi_{xx}  +2 \phi^2 - c)
\end{array} \right].
\end{equation}

When $v \neq 0$, the complexity of the phase in $Q$ makes thing somewhat more complicated.  In particular, taking $u(x,t) = (\phi(x-vt) + w(x-vt,t)) e^{i \theta ( x-vt) - i ct}$ with $w = w_1 + i w_2$, we have the linearized equation for the $(w_1,w_2)$ system as
\begin{equation}
\label{HdegNLSaltv}
\mathcal{H}_v =  \mathcal{H}_0 + v\left[  \begin{array}{cc}
- \partial_x  - \frac{\phi'}{\phi}    &  0  \\
 \frac{v}{\phi^2} &   - \partial_x  + \frac{\phi'}{\phi}  
\end{array} \right].
\end{equation}
The underlying structure of these matrix non-self-adjoint will be a topic for future work, but results analogous to those in the works of for instance Schlag et al should be possible \cite{Sch0,Sch1,Sch2,Sch3,Sch4}.

\section{Some Numerical Analysis of the degenerate NLS and KDV equations}

\label{sectionnumericalanalysis}

We consider a variation on \eqref{degNLS} with $p=4$, which comes from \cite{colliander2013behavior}.  In particular, we want to study compacton solutions for
\begin{equation}\label{eq:nls1}
- iv_t=  |v|^2 v +  \bar{v} \partial_x ( v \partial_x v) , \ \ v: \R \to \C.
\end{equation}
Here we use viscosity-type methods as motivated by the work \cite{MR2967120}, but for prior numerical works on these types of models, see works such as the use of Pade Approximants from Mihaila-Cardenas-Cooper-Saxena \cite{MR2787498} and Rosenau-Hyman \cite{HR}.

The stability of stationary solutions can be studied numerically in this equation for now in limited cases.  To handle all the cases for which we have derived solutions above, more sophisticated numerical tools will need to be developed.  Here, we treat only non-degenerate periodic solutions (when $B < 0$ or $B=0$, $A<0$) where the degeneracy of the elliptic operator does not arise generally.  However, our methods of direct numerical simulation are relatively sensitive, hence for compactons we cannot treat the cases $c \neq 0$ due to the strongly singular phase involved in generating traveling waves in the NLS setting.  To discretize \eqref{eq:nls1} in the non-degenerate case, motivated by the schemes used to solve degenerate equations in \cite{liu2017asymmetry}, we use a simple centered finite difference scheme with periodic boundary conditions.  Once we have generated the finite difference spatial operator for \eqref{eq:nls1}, we integrate in time using the stiff solver {\it ode15s} in {\it Matlab}.  The results are reported in Figure \ref{f:nlscomp}.

\begin{figure}[t]
\begin{center}
\includegraphics[width=.45\textwidth]{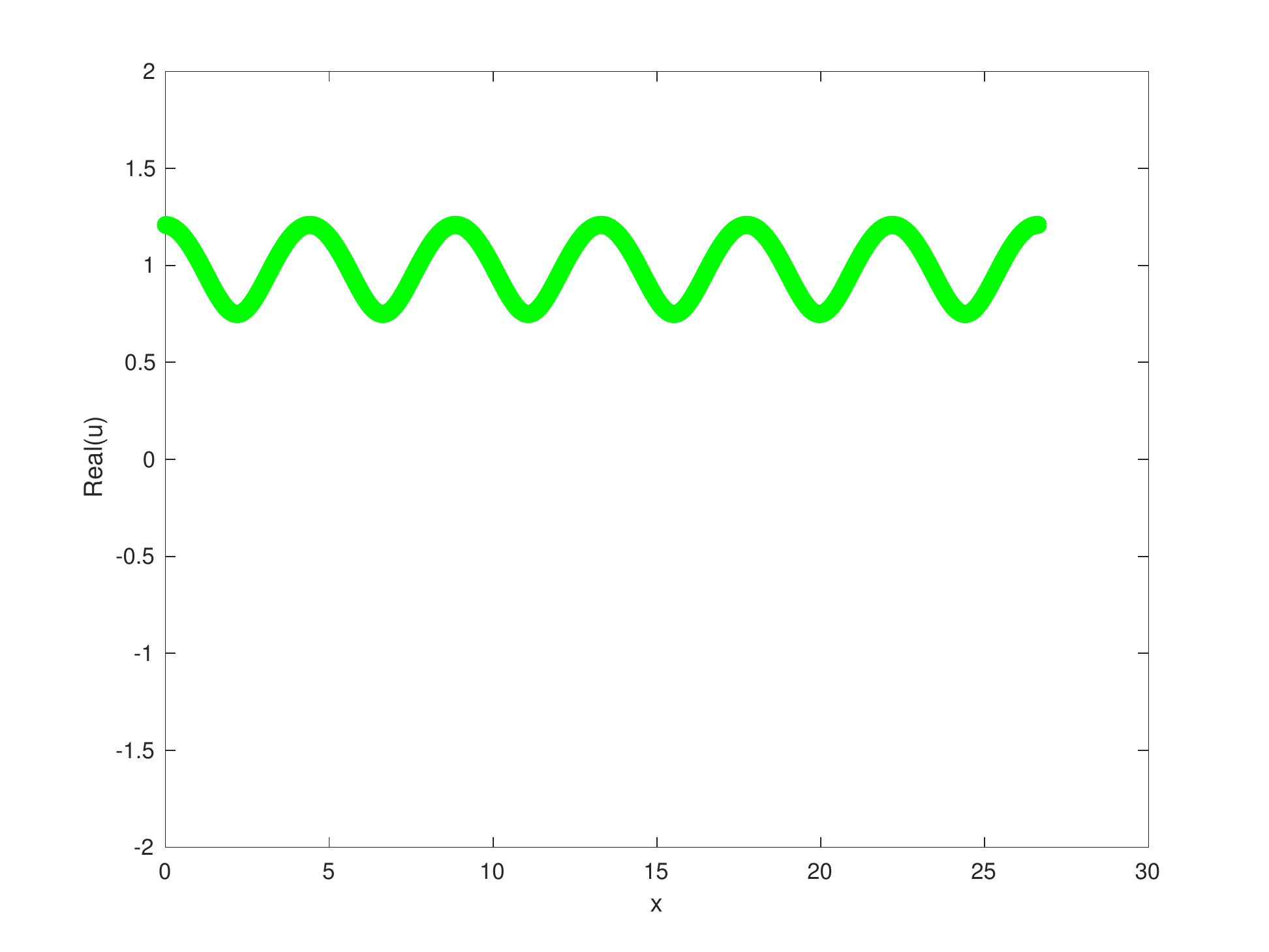}
\includegraphics[width=.45\textwidth]{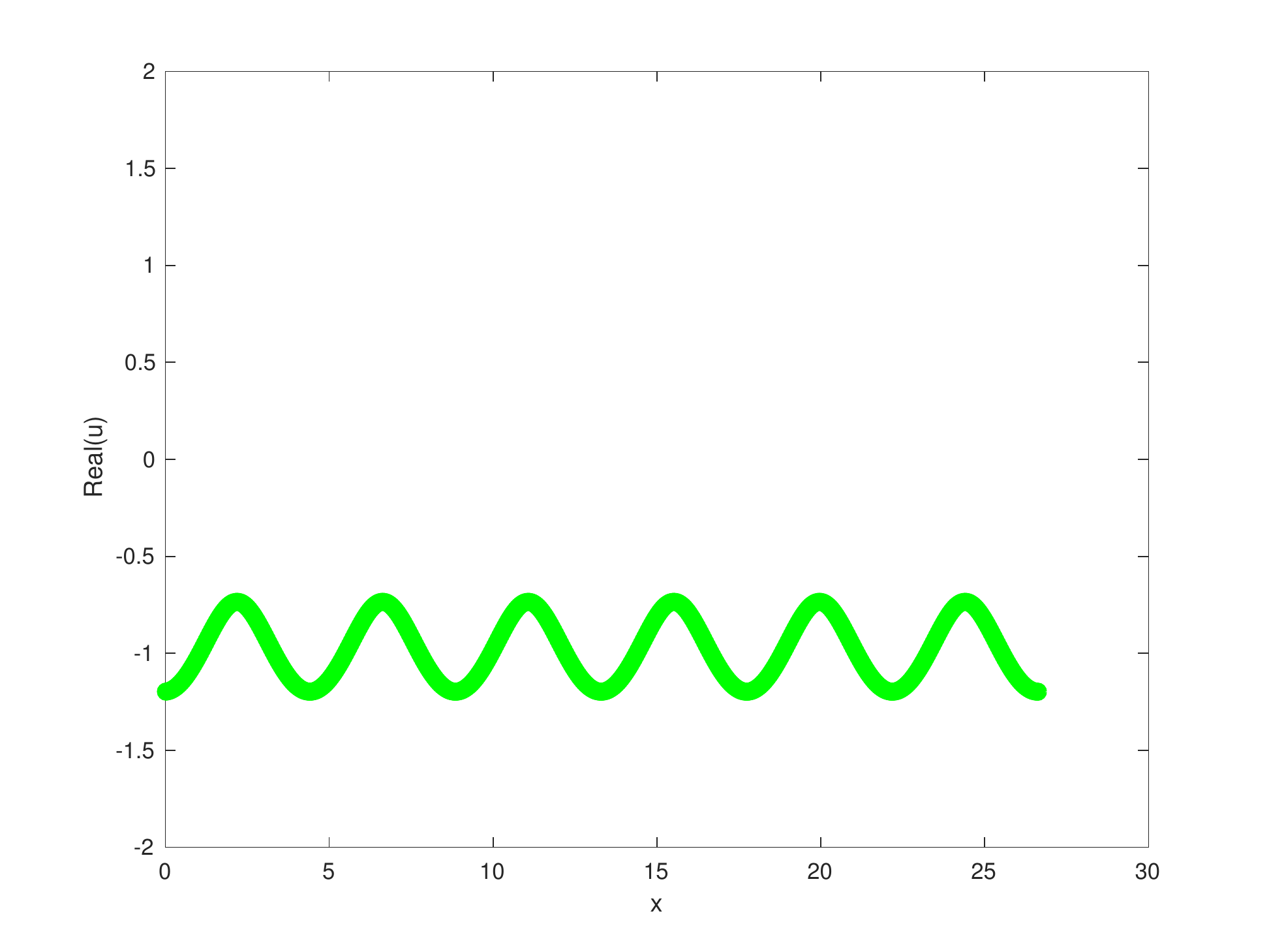}  \\
\includegraphics[width=.45\textwidth]{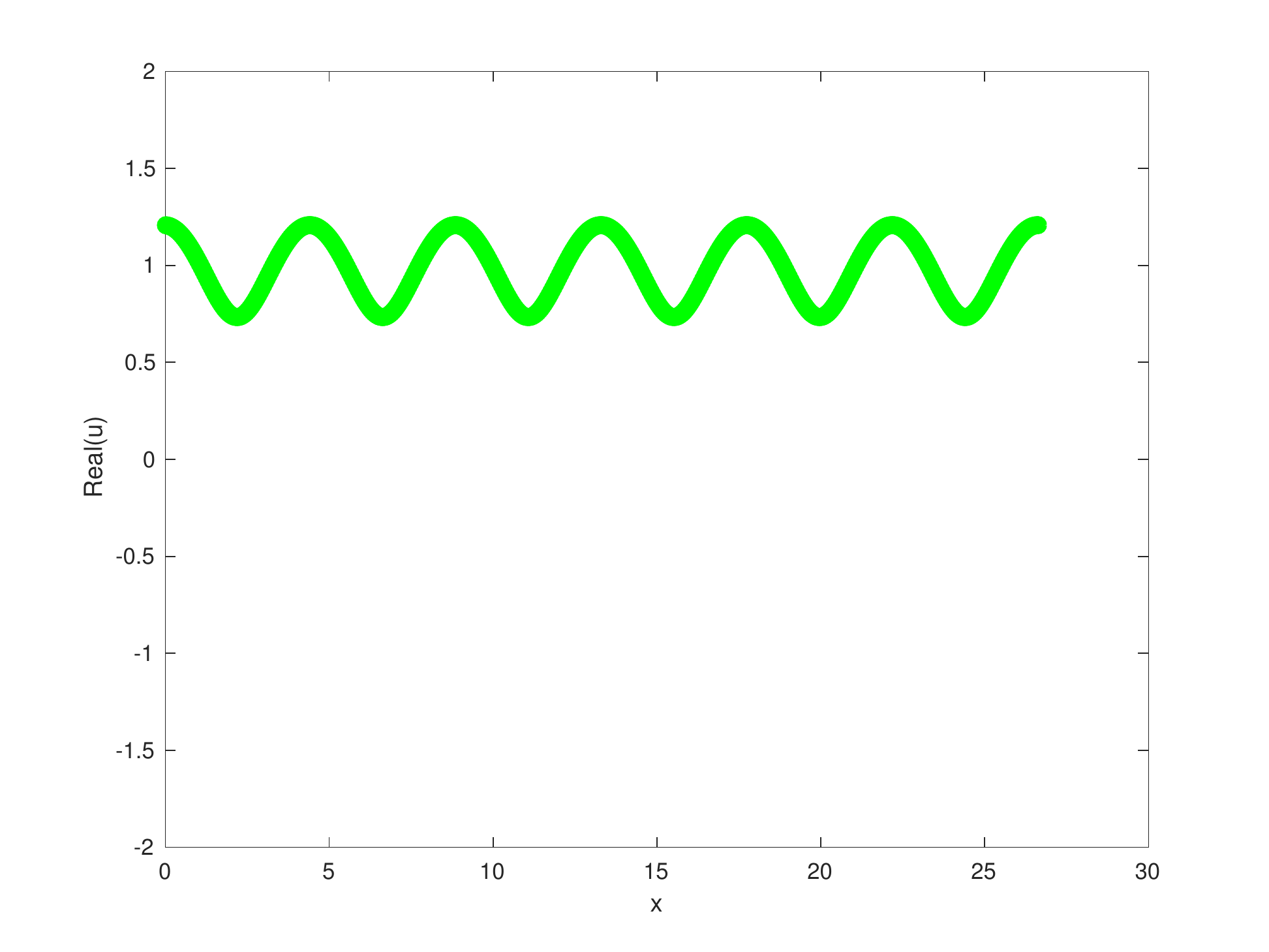}
\includegraphics[width=.45\textwidth]{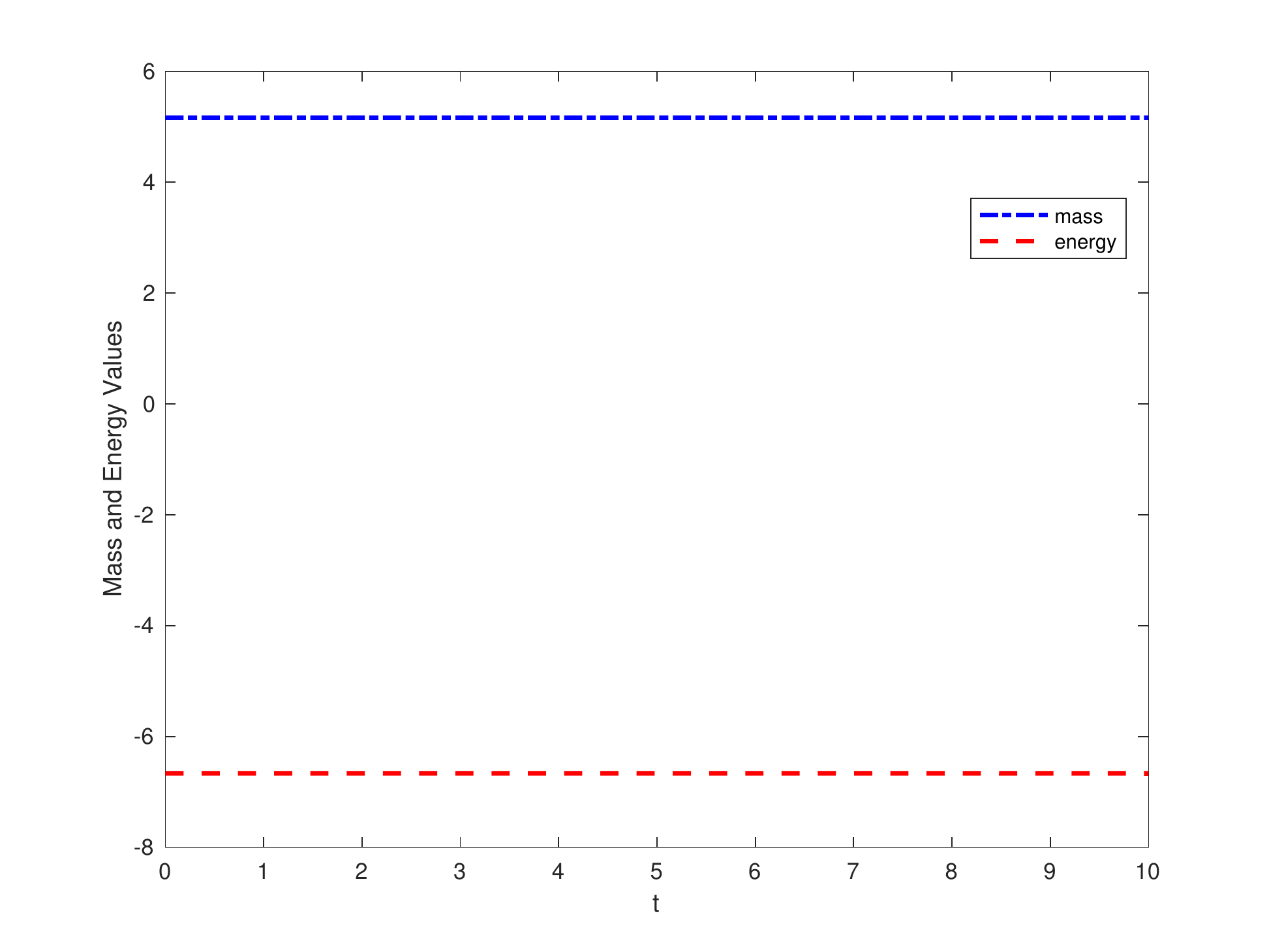}

\caption{Snapshots of the real part of solution of regularized \eqref{eq:nls1} with stationary periodic solutions ($v=0$, $c=1$, $B=-.2$) initial data at time $t=0$ (top left), time $t=\pi$ (top right), and time $t = 2 \pi$ (bottom left). Also, we plot the Mass and Hamiltonian energy over time (bottom right).}
\label{f:nlscomp}
\end{center}
\end{figure}

As an alternative,
naively taking the Madelung transformation of \eqref{eq:nls1}, $v = \sqrt{\rho} e^{i \theta}$, for this equation, we arrive at
\begin{align}
& \rho_t + 2 \partial_x ( \rho^2 \theta_x) = 0, \\
& \theta_t + 2 \rho \theta_x^2 =  \rho_{xx} +  \rho.
\end{align}
Defining $u = 2 \rho \theta_x$ as the designated flow velocity, we have
\begin{align}
\label{hydro:NLS}
& \rho_t +  \partial_x ( \rho u) = 0, \\
& u_t + 3 u \partial_x u =  \rho(  \rho_{xxx} + 2 \rho_x). \notag
\end{align}
We can numerically solve \eqref{hydro:NLS} to observe transport with relative ease by using a standard centered finite difference approximation and stiff numerical time integration schemes in {\it Matlab}.  The results are reported in Figure \ref{f:nlshydro}.

\begin{figure}[t]
\begin{center}
\includegraphics[width=.45\textwidth]{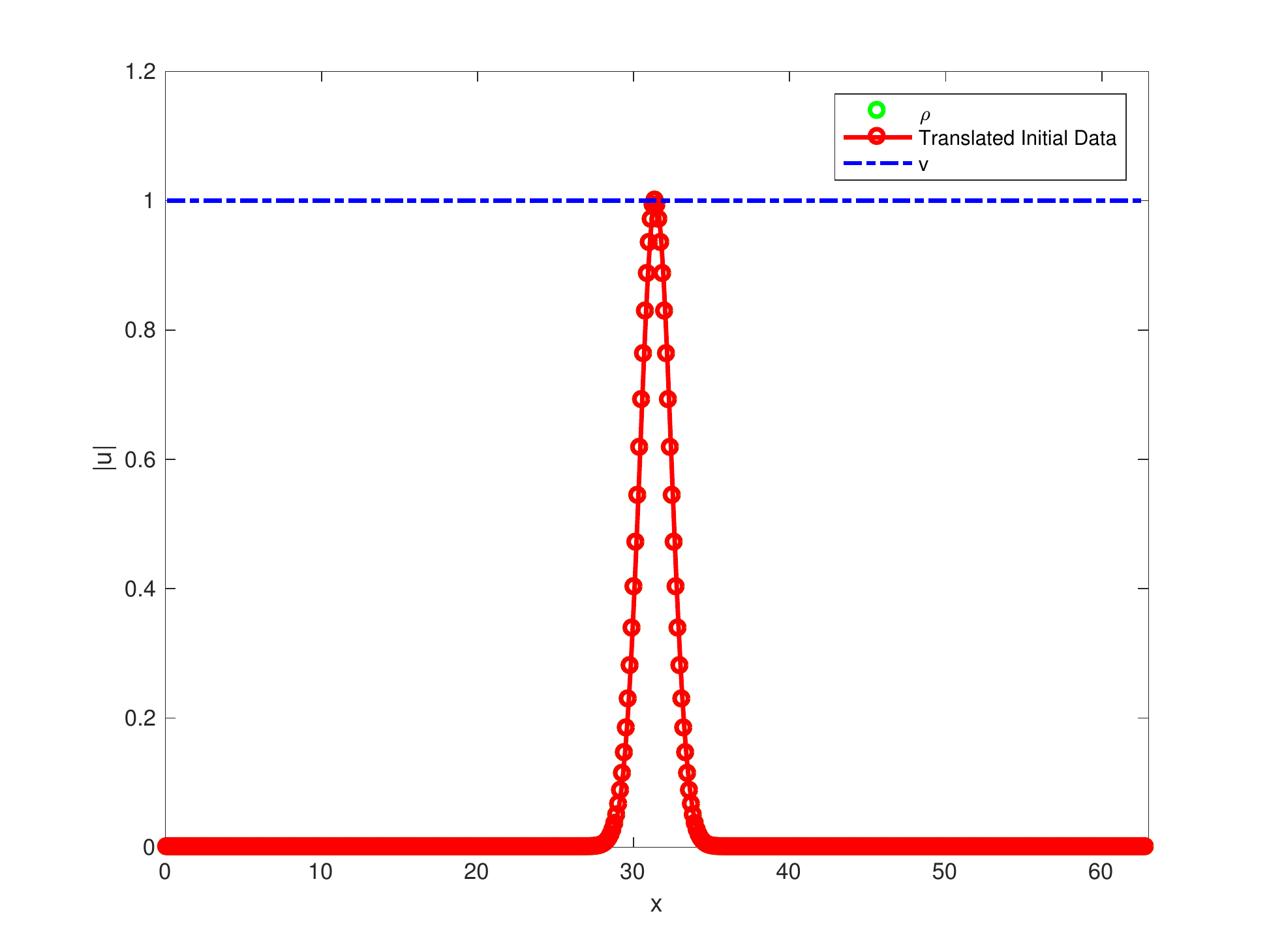}
\includegraphics[width=.45\textwidth]{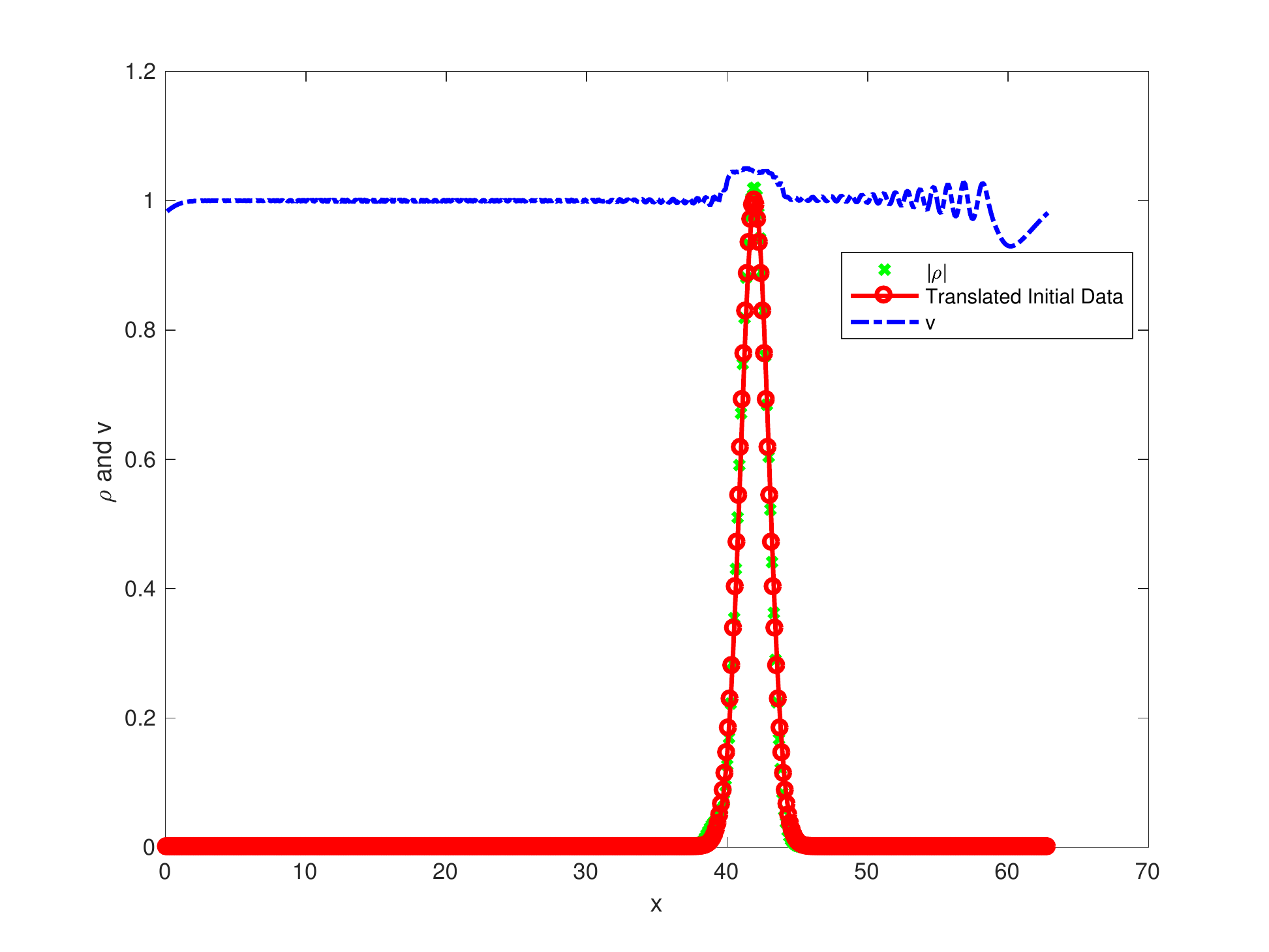} \\
\includegraphics[width=.45\textwidth]{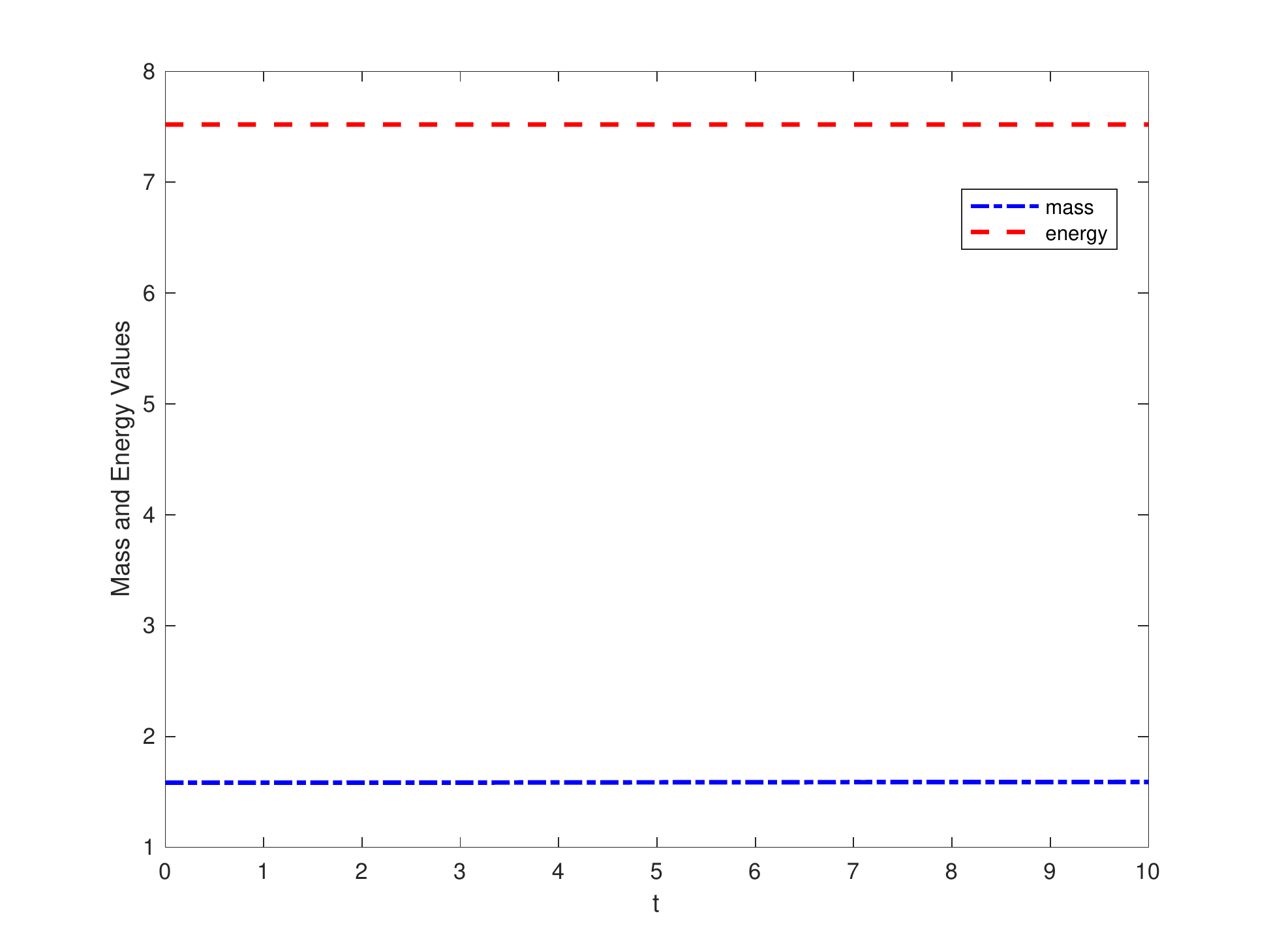}
\caption{Snapshots of the solution of regularized \eqref{hydro:NLS} with Gaussian initial data for $\rho$ and constant initial data $v_0 = 1$ at time $t=0$ (left) and time $t=10$ (right) with a comparison to the time translated initial data. Also, we plot the Mass and Hamiltonian energy over time (bottom).}
\label{f:nlshydro}
\end{center}
\end{figure}

We consider a variation on \eqref{degKdV} with $p=4$
\begin{equation}
\label{degKdV4}
\partial_t u + \partial_x (u \partial_x (u \partial_x u) + u^{3}) = 0,
\end{equation}
which has been proposed and studied in the work of Cooper-Shepard-Sodano \cite{MR1376975}.  Similar style degenerate dispersion operators have been developed by for instance Hunter-Saxon \cite{MR1135995}, etc.  To discretize \eqref{eq:nls1}, motivated by the schemes used in \cite{MR2967120}, we use a pseudospectral scheme with regularized derivatives of the form
\begin{equation}
\xi \to \frac{\xi}{1 + \nu \xi^4}
\end{equation}
for $\nu$ chosen sufficiently small (generally $\nu = 10^{-4}$ unless otherwise stated).  Once we have generated the regularized pseudospectral spatial operator for \eqref{eq:nls1}, we integrate in time using the stiff solver {\it ode15s} in {\it Matlab}.  The results are reported in Figures \ref{f:KdVcomp} and \ref{f:KdVcomppert}.

\begin{figure}[t]
\begin{center}
\includegraphics[width=.45\textwidth]{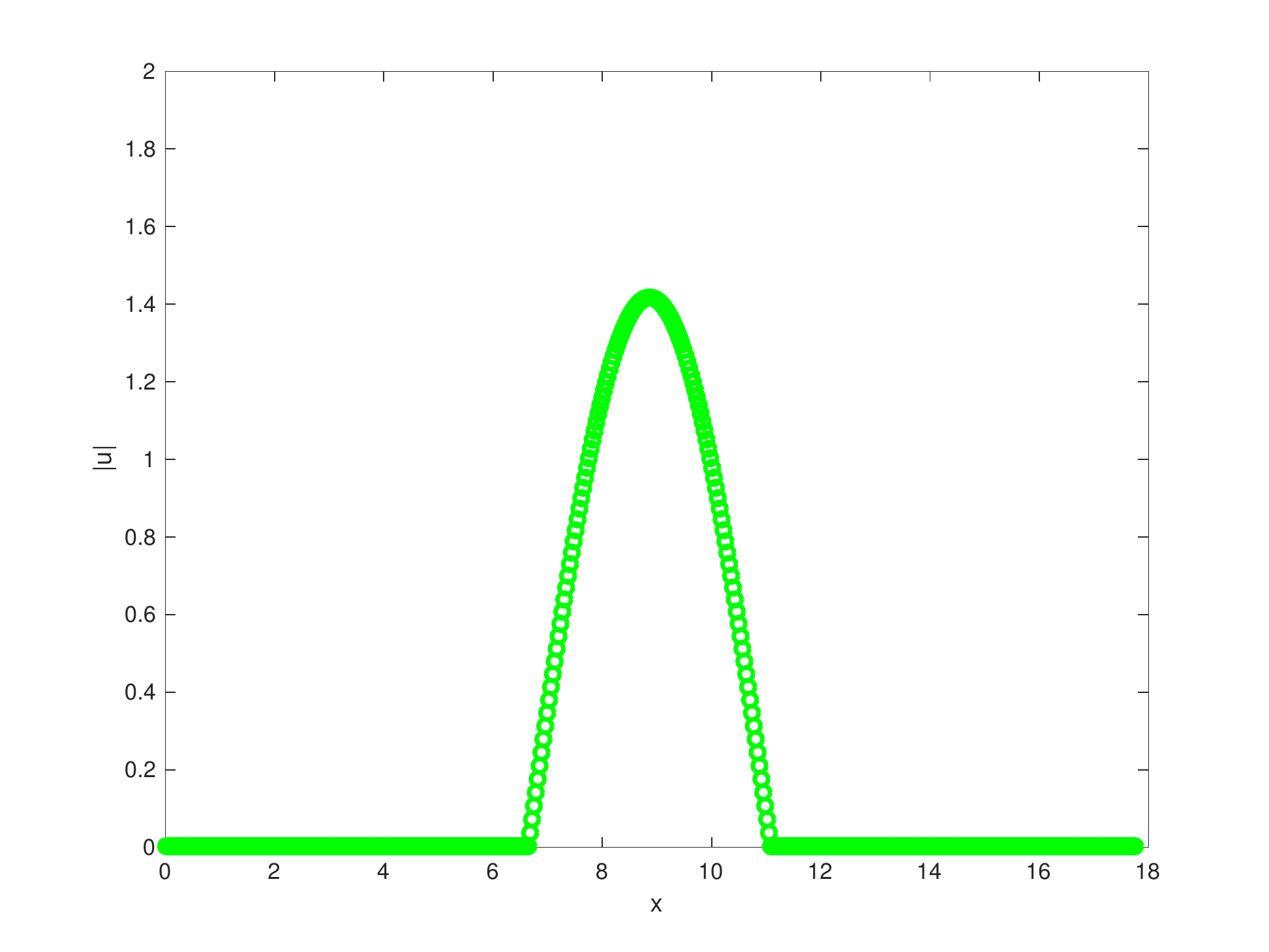}
\includegraphics[width=.45\textwidth]{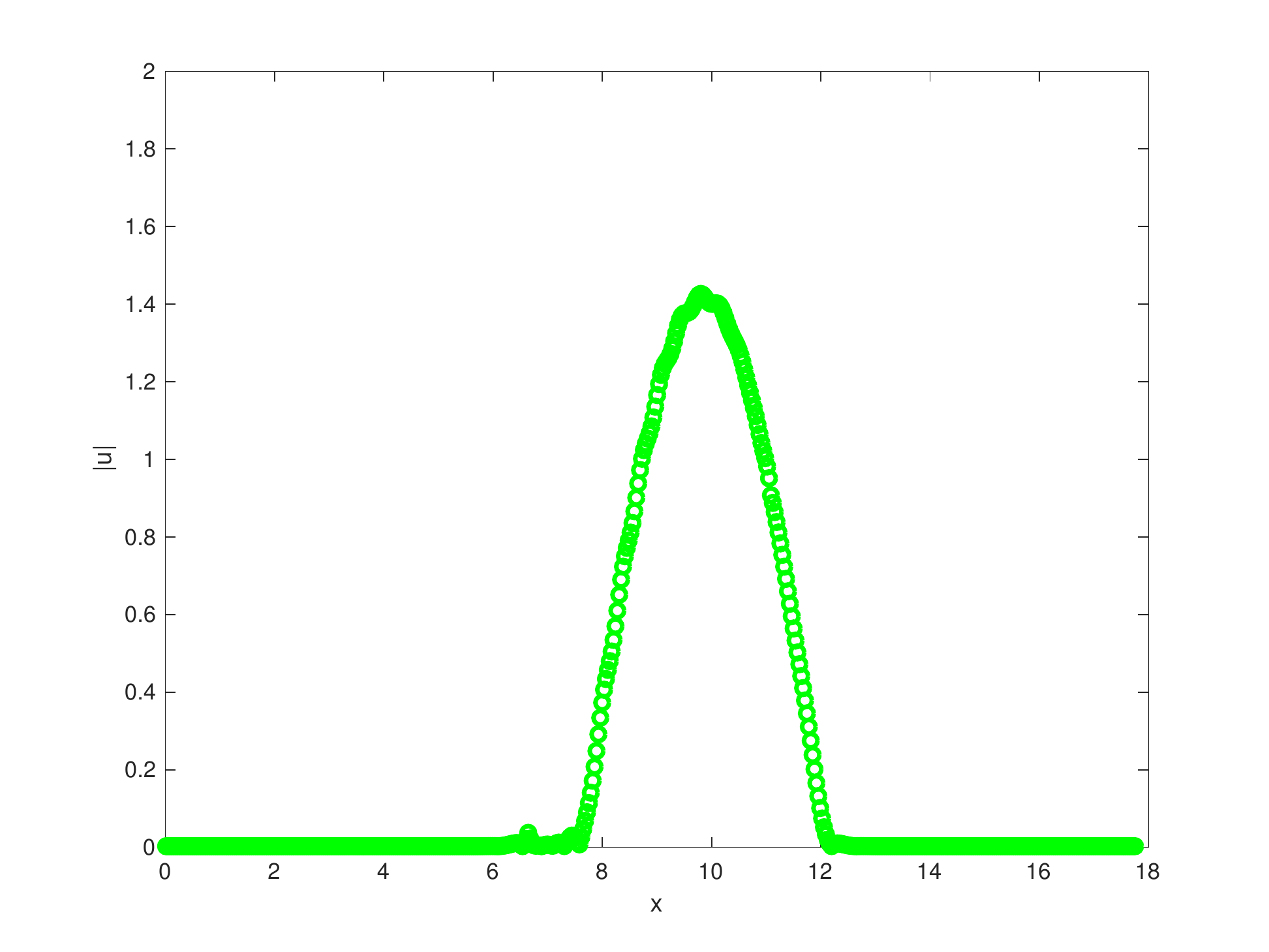} \\
\includegraphics[width=.45\textwidth]{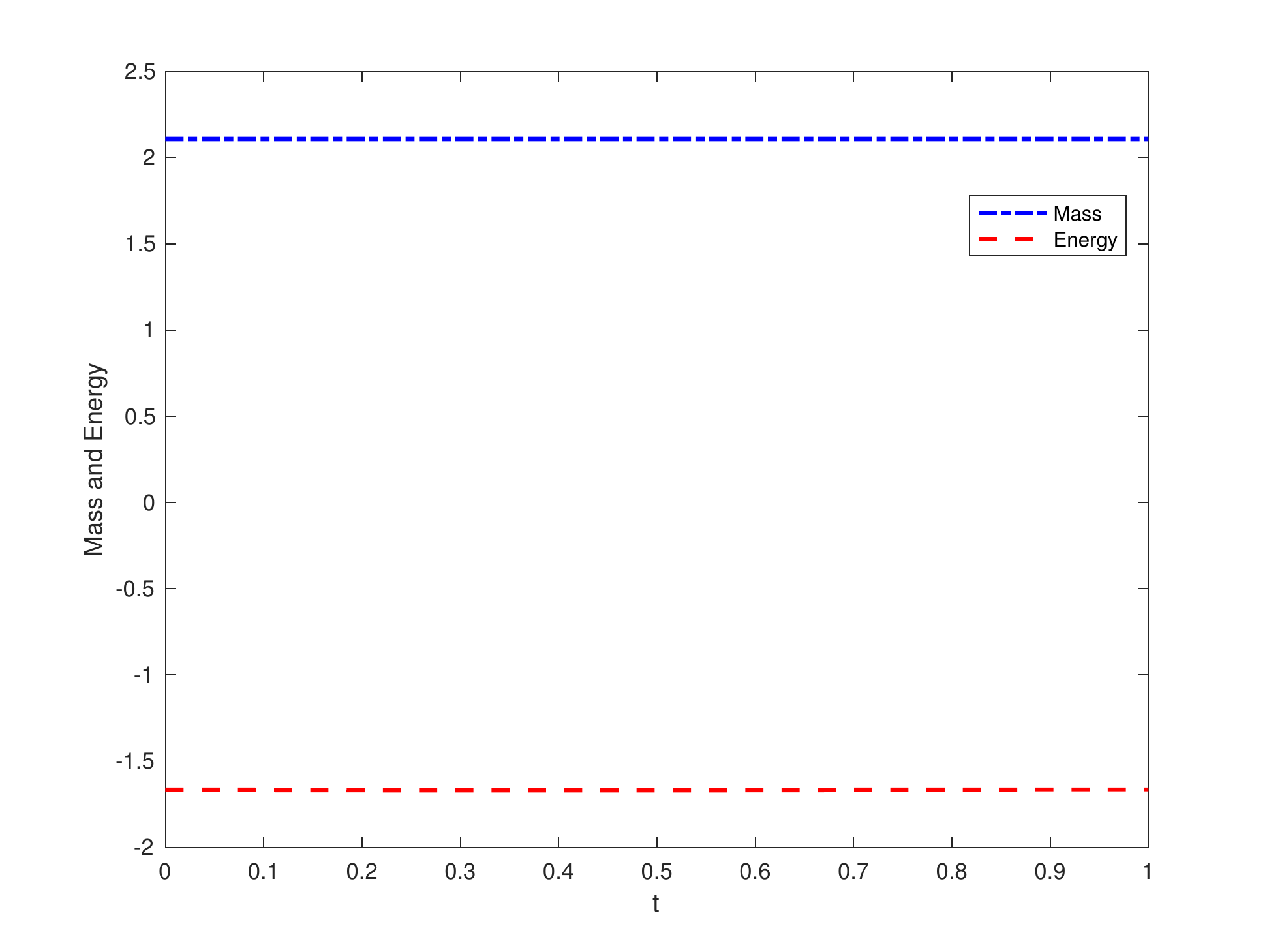}
\caption{Snapshots of the solution of \eqref{degKdV4} with  initial data given by the compacton with $c=1$, $B=0$, $A=0$ at time $t=0$ (left) and time $t=1$ (right). Also, we plot the Mass and Hamiltonian energy over time (bottom).}
\label{f:KdVcomp}
\end{center}
\end{figure}

\begin{figure}[t]
\begin{center}
\includegraphics[width=.45\textwidth]{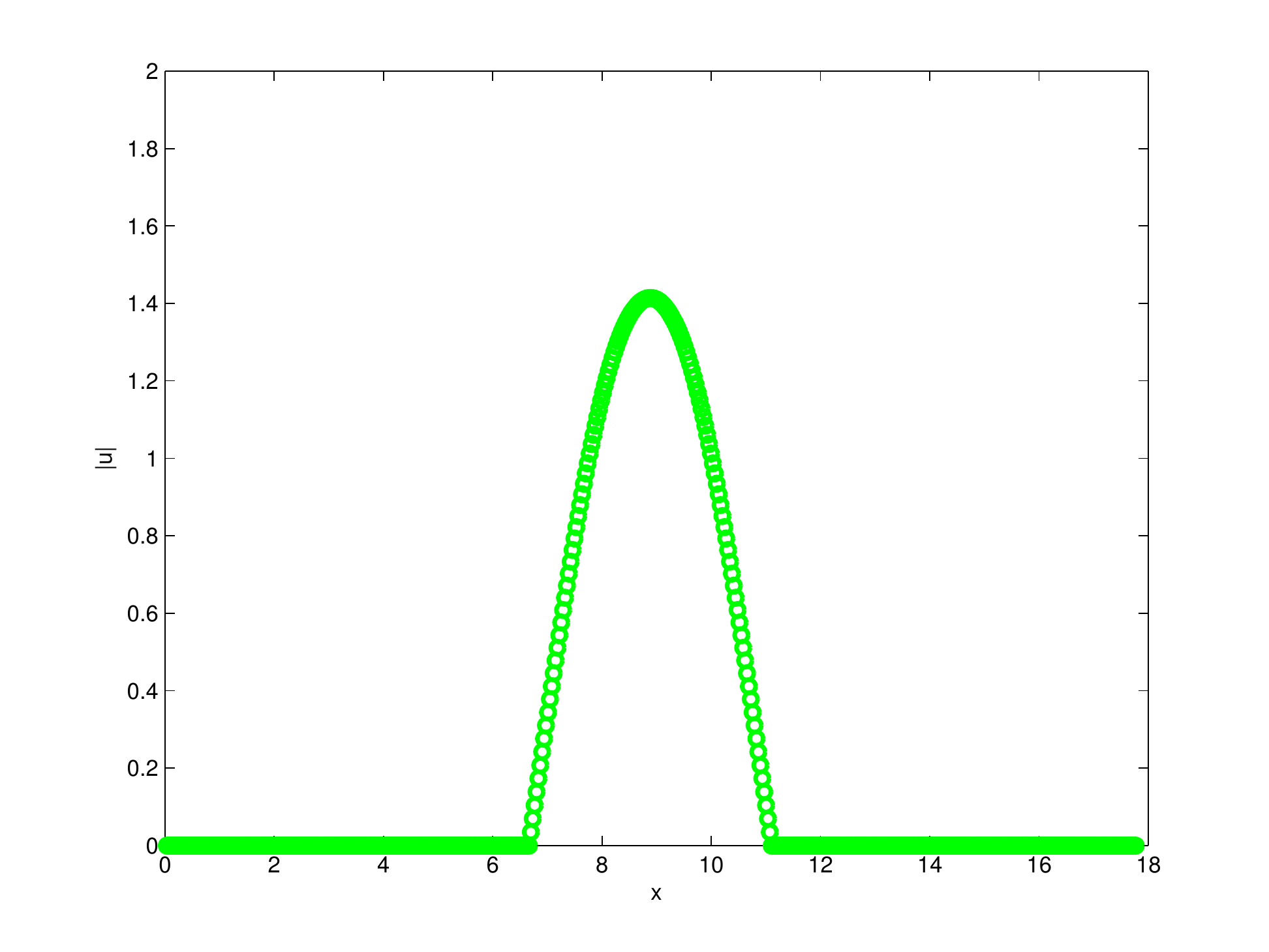}
\includegraphics[width=.45\textwidth]{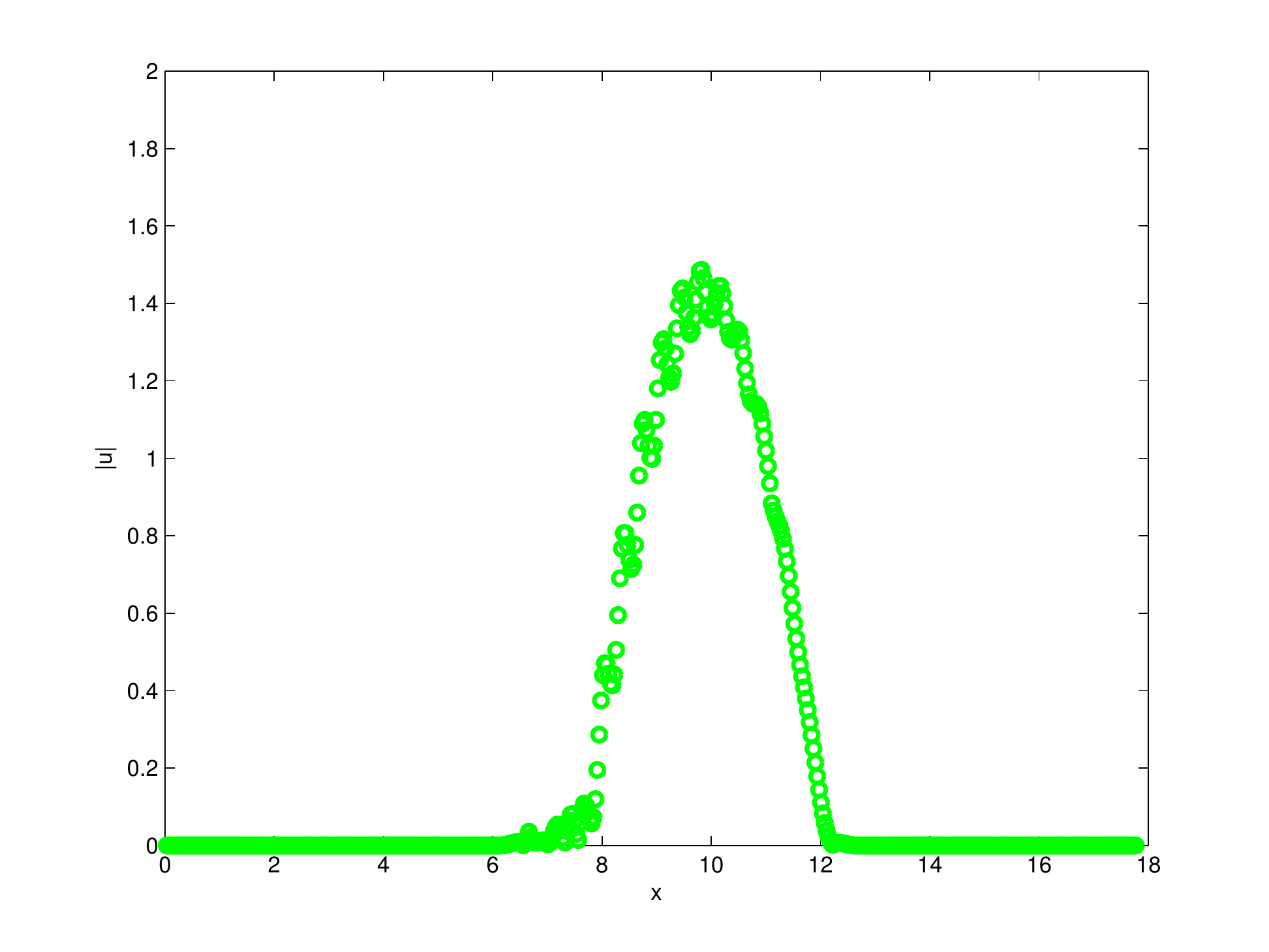} \\
\includegraphics[width=.45\textwidth]{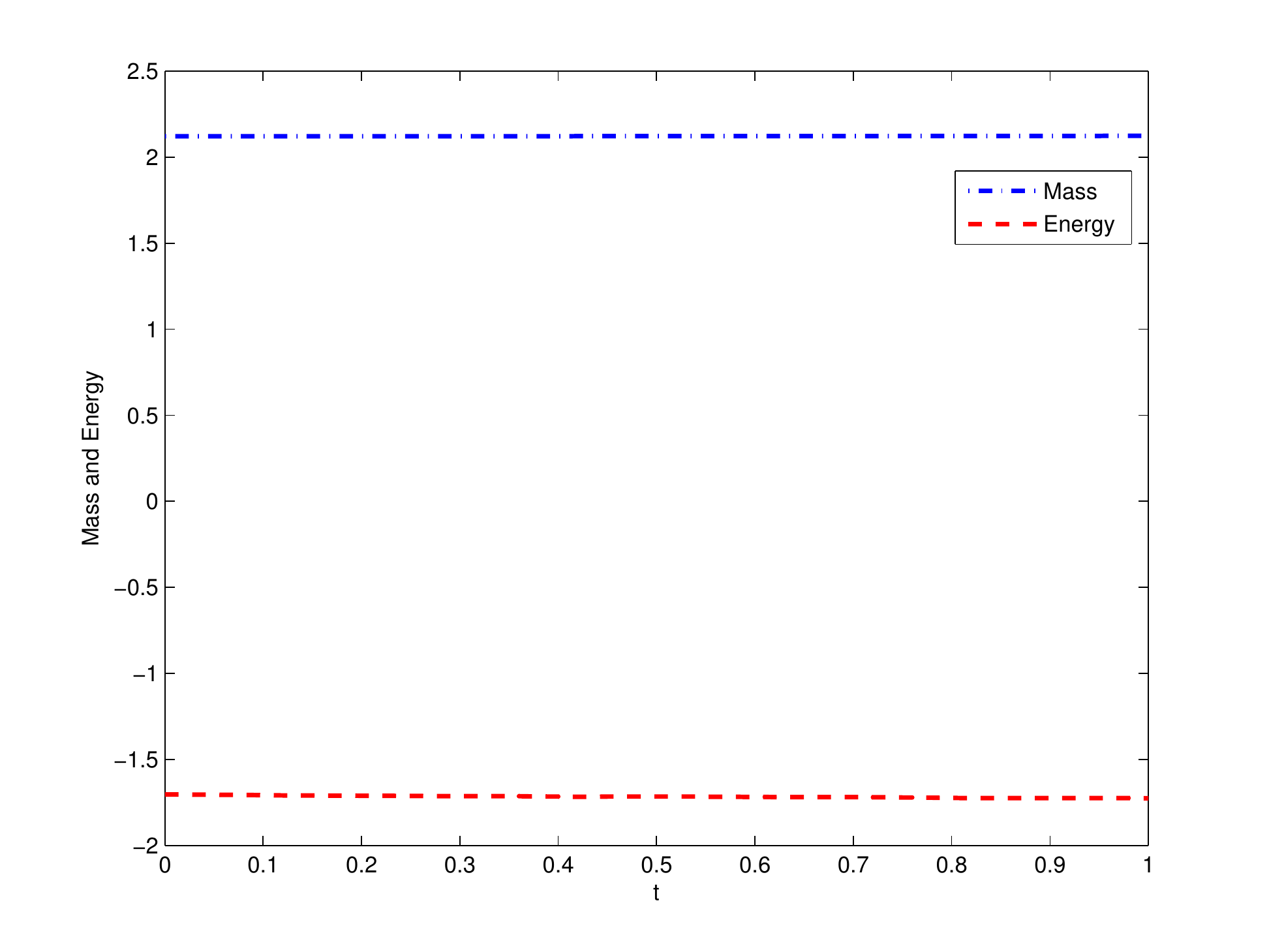}
\caption{Snapshots of the solution of \eqref{degKdV4} with  initial data given by a perturbation of the compacton, $\phi_{0,1}$, with $c=1$, $B=0$, $A=0$.  The specific initial data is $u_0 = \phi_{0,1} (x-x_0) (1 + .01 (x-x_0)^2 \phi_{0,1}^3 (x-x_0))$.  The solutions are reported at time $t=0$ (left) and time $t=1$ (right). Also, we plot the Mass and Hamiltonian energy over time (bottom).}
\label{f:KdVcomppert}
\end{center}
\end{figure}

\bibliographystyle{abbrv}
\bibliography{KdV}
\bigskip

\end{document}